\documentclass[11pt]{article}

\usepackage[latin1]{inputenc}
\usepackage[T1]{fontenc}
\usepackage[english]{babel}
\usepackage{hyperref}

\usepackage[all]{xy}
\usepackage{verbatim}
\usepackage{amsmath,amssymb}
\usepackage{amsthm}
\usepackage{graphicx}
\usepackage{epstopdf}
\usepackage{url}
\usepackage{amsfonts}
\usepackage{todonotes}
\usepackage{fullpage}
\newcommand{\mc}{\mathcal}
\newcommand{\pt}{\partial}

  \newcommand{\M}{{\mathcal M}}
\newcommand{\br}{\mathbb{R}}
\newcommand{\R}{\mathbb{R}}

\newcommand{\eps}{\varepsilon}
\newcommand{\e}{\varepsilon}

\renewcommand{\(}{\left(}
\renewcommand{\)}{\right)}
\renewcommand{\[}{\left[}
\renewcommand{\]}{\right]}

\newcommand{\bb}{\mathbb}

\renewcommand{\div}{\operatorname{div}}

\newcommand{\deb}{\rightharpoonup}
\newcommand{\destar}{\overset{*}\deb}

\DeclareMathOperator{\argmin}{argmin}

\newtheorem{thm}{Theorem}
\newtheorem{lem}[thm]{Lemma}
\newtheorem{cor}[thm]{Corollary}
\newtheorem{prop}[thm]{Proposition}
\newtheorem{defi}[thm]{Definition}
\newtheorem{remark}[thm]{Remark}

\newtheorem*{assump*}{Assumption}

\newcommand{\be}{\begin{equation}}
\newcommand{\ee}{\end{equation}}
\newcommand{\bes}{\begin{equation*}}
\newcommand{\ees}{\end{equation*}}
\def\bea{\begin{eqnarray}}
\def\eea{\end{eqnarray}}

\newcommand{\f}{{\mathcal{F}}}
\newcommand{\G}{{\mathcal{G}}}
\renewcommand{\P}{{\mathcal{P}}}

\numberwithin{thm}{section}
\numberwithin{equation}{section}

\usepackage{hyperref}
\usepackage{mathtools}
\mathtoolsset{showonlyrefs}
\newcommand*\di{\mathop{}\!\mathrm{d}}

\author{Mikaela Iacobelli\thanks{Department of Mathematical Sciences, Durham University, Durham, Lower Mountjoy, DH1 3LE, UK.\newline
Email: \textsf{mikaela.iacobelli@durham.ac.uk}}
\and
Francesco S. Patacchini\
   \thanks{Deparment of Mathematical Sciences, Carnegie Mellon University, Pittsburgh, PA 15203, USA.\newline
Email: \textsf{fpatacch@math.cmu.edu}}
   \and
   \
   Filippo Santambrogio
   \thanks{Laboratoire de Math\'ematiques d'Orsay, Univ. Paris-Sud, CNRS, Universit\'e Paris-Saclay, 91405 Orsay, France.\newline
  Email:\textsf{filippo.santambrogio@math.u-psud.fr}}
   }

\title{Weighted ultrafast diffusion equations:\\ from well-posedness to long-time behaviour}

\begin{document}
\maketitle
\begin{abstract}
In this paper we devote our attention to a class of weighted ultrafast diffusion equations arising from the problem of quantisation for probability measures. These equations have a natural gradient flow structure in the space of probability measures endowed with the quadratic Wasserstein distance. Exploiting this structure, in particular through the so-called JKO scheme, we introduce a notion of weak solutions, prove existence, uniqueness, $BV$ and $H^1$ estimates, $L^1$ weighted contractivity,
Harnack inequalities, and exponential convergence to a steady state.
\end{abstract}


\section*{Introduction}\label{sec:intro}
In this work we investigate the well-posedness and the long-time behaviour of solutions $u=u(t,x)$ of the nonlinear diffusion equation
\begin{equation}\label{eq:fde}
	\partial_t u =\div(u^{\alpha-1}\nabla u)=\Delta(u^\alpha/\alpha) \quad\mbox{on }[0,\infty)\times \Omega,
\end{equation}
where $\alpha \in \R$ and $\Omega$ is a $d$-dimensional domain; we give specific hypotheses on $\Omega$ later. This class of equations may exhibit a whole spectrum of different behaviours as $\alpha$ varies. We are interested in the case $\alpha<0$, i.e., when \eqref{eq:fde} takes the name of \emph{ultrafast diffusion} equation.
This class of equations have completely different properties from those found in the case $\alpha\geq1,$ which corresponds to the \emph{porous medium} and \emph{heat} framework. Porous medium equations model slow diffusion phenomena and have been extensively studied in the last years; we refer the reader to the monographs by J. L. V\'azquez for a comprehensive theory \cite{Va,Va2}. The case $0\leq\alpha<1$ is commonly referred to as \emph{fast diffusion} equation; in particular, $\alpha=0$ gives the \emph{logarithmic diffusion} equation \cite{DasKe}.

When $\alpha<1$, existence and uniqueness of weak solutions of the Cauchy problem, as well as the asymptotic behaviour and the main qualitative properties, are well understood when $\alpha$ lies in the so-called \emph{good parameter range} $\max\,\{0,\alpha_\mathrm{c}\}<\alpha<1$, where $\alpha_\mathrm{c} := (d-2)/d$ is a critical exponent; see for instance \cite{HP}. The theory on diffusion equations is less developed in the \emph{subcritical range} $\alpha<\alpha_\mathrm{c}$, even under the condition $\alpha>0$, since the classical questions about existence, uniqueness and regularity become more challenging. 
A typical difficulty emerging in the subcritical case concerns the possible lack of positivity due to extinction in finite time: while in the good parameter range $\alpha>\alpha_\mathrm{c}$ the mass is conserved, if we consider the case of the Cauchy problem in the whole space ${\mathbb R}^d$ with $d\ge 3$ and  $0<\alpha<\alpha_\mathrm{c}, $ P. B\'enilan and M. G. Crandall \cite{BenCra} proved the extinction in finite time of solutions of the fast diffusion equation \eqref{eq:fde} when the initial datum is in some suitable  $L^p$ space.
In fact, solutions become identically zero in finite time for all $0<\alpha<1$ if considering the Cauchy problem in a bounded domain with Dirichlet boundary data; see \cite{BGV1,BGV} for more detail. Still in $\mathbb R^d$, the \emph{critical case} $\alpha = \alpha_\mathrm{c}$ is as well very challenging since it turns out that, for $d\geq3$, solutions exhibit two space regions in which they have different long-time behaviours; see \cite{GPV}. When $\Omega$ is a bounded domain, the situation is even more involved and we refer to the monograph \cite{Va}.

Most of the literature does not treat the \emph{very singular range} $\alpha < 0$, since the diffusivity $u^{\alpha-1}$ becomes extremely singular at $u=0$. In particular, in \cite{VazNonex} V\'azquez showed that if one considers the Cauchy problem in the whole space or in a bounded domain with zero Dirichlet boundary conditions, then solutions starting from  $L^1$ initial data become instantaneously identically zero, namely $u(t)\equiv 0$ for all $t>0$.
To circumvent this phenomenon, some authors have considered initial data that are not integrable and ``not too small'' at infinity; see \cite{E,VazNonex} among the older references, then \cite{BV,DasDelP,Vazquez}, and the books \cite{DasKe,Va} for a more exhaustive discussion on the problem. It is interesting to notice that \eqref{eq:fde} with $\alpha<0$ arises naturally in certain physical applications. For example, superdiffusivities of this type have been proposed in  \cite{DG} as a model for long-range Van der Waals interactions in thin films spreading on solid surfaces. This equation also appears in the study of cellular automata and interacting particle systems with self-organised criticality; see \cite{COR} for example. Other physical applications are mentioned in \cite{BH}.

Besides the motivations above, our interest in ultrafast diffusion equations stems from the problem of \emph{quantisation for probability measures}. This problem can be stated as follows: given an integer $N$, find an atomic measure with $N$ atoms that best approximates a given probability density $\rho$ on $\Omega\subset \br^d$ in the sense of Wasserstein distances of any order $p\geq1$. As explained in \cite{GL}, this is in fact equivalent to the following minimisation problem:
\begin{equation}
\label{eq:Sigma}
	\min_{\Sigma^N}\biggl\{\int_\Omega d(x, \Sigma^N)^p\rho(x)\di x\ \  \mbox{s.t.}\ \  \# \Sigma^N=N\biggr\},
\end{equation}
where $d(x,\Sigma^N)$ stands for the distance between the point $x\in\Omega$ and the set $\Sigma^N$, which is the support of the optimal measure. Note that, in this new formulation, the only unknowns are the locations of the points of the support $\Sigma^N$. A classical and important question concerns the asymptotics of a minimiser $\Sigma^N$ when the cardinality $N$ goes to infinity. 
In order to take such a limit, one defines the probability measure $\mu^N:=\frac{1}{N}\sum_{x \in \Sigma^N}\delta_x$.
Then, it is known (see \cite{FejTot,GL,MorBol}) that as $N\to \infty$ the measures $\mu^N$ weakly-* converge to a minimiser of the energy functional
$$
	\mathcal F_\rho[f]:=\int_\Omega \frac{\rho(x)}{f(x)^{p/d}}\di x, 
$$
defined for densities $f$ on $\Omega$. This convergence has also been investigated and justified from a $\Gamma$-convergence viewpoint in \cite{BouJimRaj} (a more general proof having been established in a similar case in \cite{MosTil}; see also \cite{Iac1} and \cite{Klo} on how geometry can affect the optimal location problem).

In \cite{CGI1}, the authors introduced a new approach to the quantisation problem based on gradient flows: their idea was to  study the evolution of the points of $\Sigma^N$ when they follow the steepest descent curves of the functional \eqref{eq:Sigma} (which is nothing but a continuous-time version of the well-known Lloyd's algorithm for the optimal quantisation; see \cite{Lloyd}, or \cite{BR,Merigot} for more recent accounts and related topics), and to compare it to the gradient flow of a continuous functional.
This analysis was performed in detail in the one-dimensional case in \cite{CGI1}, and in the two-dimensional case in \cite{CGI2} when $\rho\equiv1$.
There, the authors study the Lagrangian evolution of the particles in the support of $\Sigma^N$ under the gradient flow of \eqref{eq:Sigma} and prove quantitative convergence estimates to a continuous gradient flow. As observed in \cite{Iac2}, at least when $d=1$, this continuous Lagrangian evolution of particles corresponds, in Eulerian variables, to the gradient flow of $\mathcal F_{\rho}$ in the 2-Wasserstein sense.

Because of this, understanding the $2$-Wasserstein gradient flow of $\mathcal F_{\rho}$ is a natural problem.
Before computing the equation associated to this gradient flow, we first make a short comment about the boundary conditions: as explained in \cite{CGI1}, because of the preservation of the mass in the quantisation problem, a very natural boundary condition is the no-flux (that is, Neumann) one; of course, the easiest case is actually just to suppose that the domain is a torus, which is the same as supposing that $f$ is periodic. In the sequel we shall focus on these two boundary conditions: periodic, and no-flux on bounded domains.

To compute the $2$-Wasserstein gradient flow of $\f_{\rho}$ we note that, setting $r:=p/d$, the first variation density of $\f_\rho$ at a density $f$ is given by
$$
\frac{\delta \mathcal F_{\rho}[f]}{\delta f}= -\frac{r\rho}{f^{r+1}}.
$$
Hence, by Otto's calculus (see for instance \cite{Iac2,O}) and by the theory of gradient flows in Wasserstein space (see \cite{AGS,arabsurvey}), the gradient flow of the functional $\f_\rho$ in the $2$-Wasserstein metric is given by
\begin{equation*}
	\pt_tf(t,x)=\div_x\bigg(f(t,x)\nabla_x\(\frac{\delta \mathcal F_\rho[f(t)]}{\delta f}(x)\) \bigg)=
-r\div_x\(f(t,x)\nabla_x\(\frac{\rho(x)}{f(t,x)^{r+1}}\)\).
\end{equation*}
This is an {\em ultrafast diffusion} equation weighted by the density $\rho$. Indeed when $\rho\equiv1$ it corresponds (after a change of variable and up to a multiplicative constant; see \eqref{eq:u}) to \eqref{eq:fde} with exponent $\alpha=-r<0$; this is the so-called {ultrafast diffusion regime} for which, as already explained, solutions starting from $L^1$ initial data vanish instantaneously when set on the whole space or with zero Dirichlet boundary conditions. However, the natural framework where we study this equation includes mass preservation, and thus, as mentioned above, we shall consider only periodic or Neumann boundary conditions on a bounded domain (assumed to be convex for technical reasons; see later).

As the reader will notice going on with the paper, we perform an essentially complete analysis of this weighted ultrafast diffusion equation by combining two approaches: on the one hand, we exploit as much as we can the time-discretisation given by the so-called Jordan--Kinderlehrer--Otto scheme (see \cite{JorKinOtt}) and obtain many estimates using recent tools in optimal transportation; on the other hand, we obtain further results at the level of the continuous-time PDE. Each time we choose which approach to favour depending on the easiest one to adopt.

\section{Main results and plan of the paper}
In this section we introduce the notation and assumptions, we state our main results, and we give an overview of the paper. 

Let $r$ be a positive real number. Let $\Omega\subset\R^d$ be a $d$-dimensional domain: either the $d$-dimensional \emph{torus} $\mathbb{T}^d$, or a \emph{bounded convex} domain.
Let $\rho$ be a Borel \emph{probability density} on $\Omega$, which we write either $\rho \in \P(\Omega)$ with $\rho \ll \di x$ or, abusively, $\rho \in \P(\Omega) \cap L^1(\Omega)$.

We write $\M(\Omega)$ the set of finite nonnegative Borel measures on $\Omega$, so that $\P(\Omega) = \{ \rho \in \M(\Omega) : \rho(\Omega) = 1\}$. Let us give the definition of the $2$-Wasserstein distance. For any two $\mu,\nu \in \M(\Omega)$ with same total mass, we define the $2$-Wasserstein distance $W_2(\mu,\nu)$ between $\mu$ and $\nu$ by
\begin{equation*}
	W_2(\mu,\nu) = \inf_{\pi \in \Pi(\mu,\nu)} \left( \int_{\Omega \times \Omega} |x-y|^2 \di \pi(x,y) \right)^{1/2},
\end{equation*}
where $\Pi(\mu,\nu)$ is the set of all transport plans between $\mu$ and $\nu$, that is, the subset of $\M(\Omega) \times \M(\Omega)$ consisting of measures with $\mu$ as first marginal and $\nu$ as second marginal; see \cite{OTAM,Vill} for an exhaustive account on Wasserstein metrics. 

We want to investigate the properties of the following weighted ultrafast diffusion equation discussed in the introduction:
\begin{equation} \label{eq:f}
	\partial_t f(t,x)=-r\div_x\(f(t,x)\nabla_x\(\frac{\rho(x)}{f(t,x)^{r+1}}\)\) \quad \text{on  }[0,\infty) \times \Omega,
\end{equation}
where the unknown is $f \colon [0,\infty) \to L^1(\Omega)$, with boundary conditions
$$
\frac{\partial f(t,x)}{\partial n(x)}=0 \text{ on }[0,\infty)\times\partial\Omega \quad \text{if }\Omega \text{ is a bounded convex domain in }\R^d,
$$
for all $t \in [0,\infty)$, where $n(x)$ is the outward unit normal vector to $\partial \Omega$ at $x$.
Notice that when $\Omega$ is the torus, there is no boundary condition and we can consider $f$ to be a periodic function. When $\Omega$ is a bounded convex domain, the boundary condition above should be intended in a weak sense, which means (see Definition \ref{def:weak} below) that test functions will not be compactly supported in space.

This equation (including the boundary conditions) can be seen as the gradient flow in $W_2$ of the functional
\begin{equation*}
	\mathcal F_{\rho}[f]=\int_{\Omega} \frac{\rho(x)}{f(t,x)^r}\di x, \qquad f \in L^1(\Omega);
\end{equation*}
see later for a precise definition of this functional on arbitrary measures.

Let us consider a change of variable that was first introduced in \cite{CGI1}, and that will be very useful to prove several of our estimates: for all $(t,x) \in [0,\infty) \times \Omega$,
\begin{equation}\label{eq:change f to u}
u(t,x):=\frac{f(t,x)}{m(x)}, \qquad m(x):=\rho(x)^{\frac{1}{r+1}}.
\end{equation}
With this change of variable, equation \eqref{eq:f} becomes
\begin{equation}\label{eq:u}
	\partial_t u(t,x)=-\frac{r+1}{m(x)}\div_x\bigl(m(x)\nabla_x (u(t,x)^{-r})\bigr), \qquad (t,x) \in [0,\infty) \times \Omega.
\end{equation}

In order to state our results, we need first to introduce the class of solutions on which we can prove existence and uniqueness.
Note in particular that the assumption in Definition \ref{def:weak} that initial data belong to $L^{r+3}(\Omega)$ will be used to show that weak solutions exist;
see the proof of Lemma \ref{thm:existence-weak-1}.

\begin{defi}[Weak solutions of \eqref{eq:f} and \eqref{eq:u}] \label{def:weak}
	Given $f_0 \in L^{r+3}(\Omega)$ with $\mathcal F_{\rho}[f_0]<\infty$,
we say that $f$ is a \emph{weak solution} of \eqref{eq:f} starting from $f_0$ if $t\mapsto f(t)$ is a weakly-* continuous curve valued into $\mathcal M(\Omega)$, $f(0)=f_0$, $f$ solves \eqref{eq:f} in the sense of distributions, i.e., for all $\psi \in C_\mathrm{c}^\infty((0,\infty)\times \overline\Omega)$\footnote{Note that the closure $\overline\Omega$ is compact, thus we are not imposing boundary conditions on the test functions; this is used to impose no-flux boundary conditions on the solution.} we have
\begin{equation*}
	\int_0^\infty \int_\Omega \left( \pt_t \psi(t,x) - r \nabla_x \left( \frac{\rho(x)}{f(t,x)^{r+1}} \right) \cdot \nabla_x \psi(t,x) \right) f(t,x) \di x \di t = 0,
\end{equation*}
and the following bounds hold:
$$
	\frac{f}{m} \in L^2_{\rm loc}([0,\infty),H^1(\Omega)),\qquad \Bigl(\frac{f}{m}\Bigr)^{-r} \in L^2_{\rm loc}([0,\infty),H^1(\Omega)).
$$
In this case we say that $u:=f/m$ is a weak solution of \eqref{eq:u}.
\end{defi}

The Sobolev regularity conditions in the definition are crucial. First, it is important to observe that the equation has no distributional meaning if one does not assume any Sobolev regularity on the solution. Indeed, by looking at \eqref{eq:u}, one sees that the existence of weak derivatives for $u^{-r}$ is needed to define the divergence of $m\nabla (u^{-r})$. Then, the reader will see that these precise $H^1$ assumptions play a crucial role both in the proof of uniqueness in Theorem \ref{thm:existence-uniqueness} (to make sure that we can justify the computations) and in the proof of instantaneous regularisation (or boundedness) of solutions in Theorem \ref{thm:bddness-longtime} (to be able to use the Moser iteration inspired by \cite{Mos}; see also \cite{BV,DasKe}). Before stating our main results, let us give a standing assumption on the weight $m$ which will always hold in the paper:

\begin{assump*}[Sobolev assumption on $m$]
	The weight $m$ is such that $\log m\in W^{1,p}(\Omega)$ for an exponent $p>d$. In this way $\log m$ is continuous and bounded, which means that $m$ is continuous and bounded from above and below by positive constants. When needed, we will call $\lambda$ a positive constant such that $\lambda<m<\lambda^{-1}$. 
\end{assump*}

Although this assumption always implicitly holds, we will recall it in some results to emphasise its importance. Note that, because of this hypothesis on $m$ (and thus on $\rho$), the domain of the functional $\f_\rho$ is exactly the subset of densities $f$ such that $1/f \in L^r(\Omega)$. Throughout the paper, it will be sometimes necessary to assume extra properties on $\log m$ (in particular, semiconcavity), and to assume bounds on $f_0$. Since these properties and bounds are not always the same in every result we state later, we prefer not to detail them here but rather provide them whenever we need.

\begin{thm}[Existence and uniqueness]\label{thm:existence-uniqueness}
Suppose that $f_0 \in L^{r+3}(\Omega)$ with $\mathcal F_{\rho}[f_0]<\infty$. Then there exists a unique weak solution of \eqref{eq:f} starting from $f_0$.
\end{thm}

With this theorem in hand, we can address the regularity properties of solutions and their long-time behaviour. In \cite{Iac2} the author showed that when $\rho$ is a positive smooth function and one considers smooth solutions of \eqref{eq:f} starting from initial data bounded away from zero and infinity, then as time goes to infinity the solution converges exponentially fast to the stationary state
\bes
	x \mapsto M \gamma m(x), \qquad M:=\int_\Omega f_0, \quad \gamma:=\left(\int_\Omega m \right)^{-1},
\ees
so that $f$ converges to a stationary state with the same mass\footnote{In the literature related to the quantisation problem, $f_0$ is a probability measure \cite{CGI1, Iac2}; on the other hand, in the literature about fast-diffusion equations the mass is arbitrary, and often nonpreserved. For the sake of generality we admit here arbitrary masses and arbitrary initial data in $L^1(\Omega)$.} as $f_0$. We recover in this paper the same convergence as in \cite{Iac2} without the initial boundedness condition, since we prove instantaneous upper and lower bounds (usually called Harnack inequalities) beforehand. Although upper bounds are rather classical in these settings, lower bounds are nonstandard and actually false in many situations \cite{VazNonex}. In our case it is crucial that we work with periodic or no-flux boundary conditions.

As we shall see, this result is crucial for the long time behaviour because, once the solution is bounded and bounded away from zero, the singular/degenerate character of the equation is not predominant and the solutions behave like standard parabolic equations.

\begin{thm}[Harnack inequalities]\label{thm:bddness-longtime}
	Suppose that $f$ is a weak solution of \eqref{eq:f} starting from some density $f_0$. Assume the following integrability properties on $f_0$: we have $\int_\Omega ( f_0^{p_+}+f_0^{-p_-})<\infty$, where $p_+\geq r+3, \;p_+>(r+1)d/2,\; p_->(r+1)\max\(1,d/2\)$. Then, for any $t>0$ there exists a  constant $C_t>0$ (nonincreasing in $t$) such that
\bes
	\frac{1}{C_t}\leq f(s)\leq C_t\qquad \text{for all $s \geq t$}.
\ees
\end{thm}

We can then apply the above theorem, together with $H^1$ and $BV$ estimates that will be proven later, in order to obtain the following:

\begin{thm}[Long-time behaviour]\label{thm:bddness-longtime2}
	Let $f$ be a weak solution of \eqref{eq:f} starting from some $f_0$ satisfying the same assumptions as in Theorem \ref{thm:bddness-longtime}. Then, there exist constants $C,c>0$ such that, for all $t\geq0$, one has 
$$
	\left\|f(t) - M\gamma m\right\|_{L^2(\Omega)}  \leq Ce^{-ct}.
$$
Moreover, if one adds the assumption that there exists $\Lambda\in\R$ such that $D^2(\log m) \leq \Lambda\, {\rm Id}$, then there exists another constant $C>0$ so that one also has
$$
	\left\|f(t) - M\gamma m\right\|_{BV(\Omega) } \leq Ce^{-ct} \qquad \text{for $t$ large enough}.
$$
\end{thm}

We now briefly explain the ideas behind the proofs of the above results. To prove existence of weak solutions we use the so-called Jordan--Kinderlehrer--Otto (JKO) scheme. This method, first introduced in \cite{JorKinOtt} and further developed in several other papers (see for instance \cite{DiFMatt14} for a related setting), provides a very natural way to discretise Wasserstein gradient flows in time. More precisely, given a time step $\tau>0$, one fixes $f^{(\tau)}_0:=f_0$, and, for each $k\in\mathbb{N}$, one defines  $f^{(\tau)}_{k+1}$ as the minimiser of the functional
$$
f\mapsto \mathcal F_{\rho}[f]+\frac{W_2^2(f,f^{(\tau)}_k)}{2\tau}.
$$
In this way one constructs a discrete gradient flow defined at all times $t=k\tau$ with $k\geq 0$, and to obtain a solution to \eqref{eq:f} one needs to find a limit as $\tau \to 0$. In our case we face at least two main challenges: first, the JKO scheme is naturally set in the class of measures, and we would need to prove that minimisers of the above functional exist in the space of functions, or densities (a priori, the minimiser may have a singular part); second, we need to prove enough estimates on the discrete solutions to ensure that, in the limit, we obtain a weak solution according to Definition \ref{def:weak}. To circumvent these difficulties, we first prove that if the initial datum $f_0$ is bounded between two multiples of $m$, then the same bound is true for $f^{(\tau)}_k$ for all $k\in\mathbb{N}$. In this way we guarantee that $f^{(\tau)}_k$ is a function (and not only a measure). Also, still assuming that $f_0$ is bounded between two multiples of $m$, we exploit the so-called ``flow-interchange technique'' (see \cite{MattMcCSav}) and the so-called ``five-gradients inequality'' (see \cite{DePMesSan}) to find $H^1$ and $BV$ a priori estimates on our discrete solutions. In this way we can prove the existence of a weak solution of \eqref{eq:f} whenever $0<c_0m \leq f_0\leq C_0m$ for some $c_0,C_0>0$. Finally, the general existence theorem follows by approximation. 

To prove uniqueness the idea is to consider two weak solutions $f$ and $g$, and show that 
$$
	t\mapsto \int_\Omega (f(t)-g(t))_+
$$
is descreasing in time. To achieve this we use the equation satisfied by $u:= f/m$ and $v:= g/m$; see \eqref{eq:u} and we prove a \emph{weighted} $L^1$ contraction on $u$ and $v$. Note that we are actually unable to prove directly this property for all solutions: we can prove it only when one of the two solutions is uniformly bounded away from zero; see Proposition \ref{prop:contr}. Then, by an approximation argument, we are able to conclude the desired uniqueness; see Theorem \ref{thm:uniq sol}. Finally, we prove the instantaneous positivity and boundedness from above (i.e., instantaneous regularisation) of weak solutions using a Moser iteration, and then we conclude the $L^2$ exponential convergence relying on the argument in \cite{Iac2}. We are also able to provide $BV$ exponential convergence to the steady state, using the arguments deriving from the discrete $BV$ estimate.

In Section \ref{sec:time-disc} we discretise the problem in time and show existence of minimisers for the JKO scheme, together with a discrete maximum principle; we also give $BV$ and $H^1$ estimates for the minimisers. In Section \ref{sec:existence-uniqueness} we prove Theorem \ref{thm:existence-uniqueness}; as a corollary, we also get a continuous maximum principle. Then, in Section \ref{sec:reg} we show Theorem \ref{thm:bddness-longtime} and in Section \ref{sec:long-time-behaviour} we use our tools to prove exponential convergence, that is, Theorem \ref{thm:bddness-longtime2}. Note that, whenever relevant, we rewrite our results in remarks for the nonweighted case $\rho\equiv1$, which is the prototype equation and helps understand the essential aspects of both the problem and the results.

\section{Time discretisation of the problem}\label{sec:time-disc}

Let us fix in this section the time step $\tau>0$.

In order to study the JKO scheme, we first define our functional $\mathcal F_\rho$ on the space of measures. To this aim, for all $\mu \in \M(\Omega)$ that we can decompose as $\mu=f\di x+\mu^\mathrm{s},\; \mu^\mathrm{s} \perp \di x$,  we define
$$
\mathcal F_{\rho}[\mu]= \int_\Omega \frac{\rho}{f^r}\di x.
$$
Note that, if we set $U(s)=s^{-r}$ for all $s \in (0,\infty)$ we can also define the functional
$$
	\G [\mu]=\int_\Omega U\left(\frac{f}{m}\right)m \di x,
$$
where $m$ is as in \eqref{eq:change f to u}. Of course, we have $\mathcal F_{\rho}[\mu]=\G [\mu]$. More generally, and for future use, for a given exponent $q<0$, and still using the decomposition $\mu=f\di x+\mu^\mathrm{s},\; \mu^\mathrm{s} \perp \di x$, we define 
$$
	\G_{(q)} [\mu]=\int_\Omega \left(\frac{f}{m}\right)^qm \di x,
$$
and we also give a similar, but different definition, for $q>1$:
$$
	\G_{(q)} [\mu]=\begin{cases}\int_\Omega \left(\frac{f}{m}\right)^qm \di x,&\text{if }\mu^\mathrm{s}=0,\\
							+\infty&\text{if not.}\end{cases}
$$
Also note that, when the reference weight $m$ is not fixed (for instance, we will once use a sequence of weights $m_n$), then we can also write $\G_{(q;m)}$ instead of $\G_{(q)}$ to stress the dependence on the weight.

The functional $\G$ previously defined is just an example of functional $\G_{(q)}$, for $q=-r$. 
 We observe that all functionals $\G_{(q)}$  are lower semicontinuous for the weak-* convergence of nonnegative measures because they have the form $\mu\mapsto \int_\Omega U(\di \mu/\di x)\di x+ U'_\infty \int_\Omega \mu^\mathrm{s}$, where $U'_\infty := \lim_{s \to +\infty} U(s)/s$ (and here we have $U'_\infty\in\{0,+\infty\}$); see for instance \cite[Proposition 7.7]{OTAM}. (Note that if $U_\infty'=+\infty$ and $\int_\Omega \mu^\mathrm{s}=0$, then we conventionally set $U_\infty' \int_\Omega \mu^\mathrm{s} = 0$.)

As explained in the introduction, one can construct a discrete gradient flow as an iterative sequence of minimisation problems of the form
$$
	\mu_\tau^k \text{ minimises }\mu\mapsto \mathcal \G[\mu]+\frac{W_2^2(\mu,\mu^{(\tau)}_{k-1})}{2\tau},
$$
for every $k \in \mathbb{N}$. This means that, for a given $\nu \in \M(\Omega)$ with mass $M$, we want to solve
\begin{equation}\label{minGFnu}
	\min\left\{\G[\mu]+\frac{W_2^2(\mu,\nu)}{2\tau}\;:\;\mu\in\mc{M}_M(\Omega)\right\},
\end{equation}
where $\mc{M}_M(\Omega):=\{\mu \in \M(\Omega) \mbox{ s.t.}\, \int_{\Omega}\di\mu=M\}$.

Note that, as a consequence of the definition of the scheme, the mass $M$ of our discrete solutions (and therefore also of their continuous limits) is preserved.

\subsection{Well-posedness of the discrete scheme}

\begin{thm}\label{exiuniqJKO}
	If $\nu \in \M(\Omega)$ with $\nu \ll \di x$, then there exists a unique minimiser for Problem \eqref{minGFnu}
\end{thm}

\begin{proof}
	The functional $\G$ is lower semicontinuous for the weak-* convergence of measures, and so is $\mu\mapsto W_2^2(\mu,\nu)$, since $W_2$ exactly metrises (on compact sets) this convergence. Moreover, the set $\M(\Omega)$ is compact for this convergence, which proves the existence of a minimiser. Uniqueness comes from the strict convexity of the problem. Indeed, $\G$ is a convex functional and so is $\mu\mapsto W_2^2(\mu,\nu)$. In addition, the latter is also strictly convex if $\nu\ll \di x$; see \cite[Proposition 7.19]{OTAM}.
\end{proof}

Theorem \ref{exiuniqJKO} does not exclude the possibility of a minimiser, say $\mu_*$, which is not absolutely continuous: its singular part $\mu_*^\mathrm{s}$ does not enter into play in the computation of $\G [\mu_*]$ but is not forbidden as soon as the absolutely continuous part of $\mu_*$ is positive almost everywhere. In order to study the minimisers for Problem \eqref{minGFnu} we approximate $U$ with a superlinear cost function to ease the computations. Define, for all $\eps >0$,
$$
	U_\eps(s):=s^{-r}+\eps \frac{s^2}{2} \quad \mbox{ for all } s \in (0,\infty),
$$
and, for all $\mu \in \M(\Omega)$,
$$
	\G_\e[\mu] = \begin{cases}\int_\Omega U_\eps\left(\frac{f}{m}\right)m& \mbox{ if }\mu= f \di x\mbox{ with } f \in L^1(\Omega),\\
						+\infty&\mbox{ if not.}\end{cases}
$$
We use the following result, which is essentially a statement on $\Gamma$-convergence; see for instance \cite{DalMaso}:
\begin{lem}\label{approxUeps}
	Suppose $\nu \in \M(\Omega)$ with $\nu\ll \di x$ and $\int_\Omega d\nu=M$. Given $(\nu_\eps)_\eps$ weakly-* converging to $\nu$, set $M_\eps:=\int_\Omega \di \nu_\eps$ (in particular, we have $M_\eps\to M$ as $\eps\to0$); then, the problem 
\be\label{eq:min-approx}
	\min\left\{\G_\eps[\mu]+\frac{W_2^2(\mu,\nu_\eps)}{2\tau}\; : \;\mu\in\mc{M}_{M_\eps}(\Omega)\right\}
\ee
admits a unique solution $\mu_\eps$ for every $\eps>0$. This solution is absolutely continuous for every $\eps>0$, and weakly-* converges to the unique solution $\mu_*$ of \eqref{minGFnu} as $\eps \to 0$.
\end{lem}
\begin{proof}
	Given $\eps>0$, the existence and uniqueness of $\mu_\eps$ can be done as in the proof of Theorem \ref{exiuniqJKO}. (Uniqueness is actually easier since the functional $\G_\eps$ is strictly convex, so that we do not need the strict convexity of the Wasserstein part and we do not need $\nu_\eps\ll \di x$.) The fact that $\mu_\eps$ is absolutely continuous is straightforward, since otherwise $\G_\eps[\mu_\eps]=+\infty$. Up to subsequences, we can always suppose $\mu_\eps\destar\mu_*$ as $\eps\to0$ for some $\mu_* \in \M_M(\Omega)$; indeed, $\mu_*(\Omega) = \nu(\Omega)$ since the weak-* convergence preserves in this case the total mass \cite{Bill}. We now just need to prove that $\mu_*$ solves \eqref{minGFnu}. Given an arbitrary measure $\tilde\mu$ with an $L^2$ density, we can write
$$
	\G[\mu_\eps]+\frac{W_2^2(\mu_\eps,\nu_\eps)}{2\tau}\leq \G_\eps[\mu_\eps]+\frac{W_2^2(\mu_\eps,\nu_\eps)}{2\tau}\leq \G_\eps[\tilde\mu]+\frac{W_2^2(\tilde\mu,\nu_\eps)}{2\tau}.
$$
Passing to the liminf as $\eps\to 0$, using the semicontinuity of $\G$, the continuity of $W_2$ with respect to the weak-* convergence and the fact that we have $\G_\eps[\tilde\mu]\to \G[\tilde\mu]$, which is true for every $\tilde\mu\in L^2$ (since the extra term in the definition of $\G_\eps$ is a finite term multiplied by $\eps$), we get
$$
	\G[\mu_*]+\frac{W_2^2(\mu,\nu)}{2\tau}\leq \G[\tilde\mu]+\frac{W_2^2(\tilde\mu,\nu)}{2\tau}.
$$
This shows that $\mu_*$ is a minimiser in  \eqref{minGFnu} if we restrict to $L^2$ competitors.

To complete the proof, it is enough to prove that the infimum in  \eqref{minGFnu} does not change if we restrict it to $L^2$, or, in fact, even to bounded, densities. To do so, take an arbitrary $\mu=f\di x+\mu^\mathrm{s}$ and define, for all $p>0$, $\mu_p:=\(c_p+(f\wedge p)+f_p\)\di x$, where $f \wedge p$ stands for the minimum between $f$ and $p$, $f_p\destar\mu^\mathrm{s}$ as $p\to\infty$ is an arbitrary $L^\infty$ approximation of $\mu^\mathrm{s}$, and $c_p=|\Omega|^{-1}\int_\Omega (f-f\wedge p)\di x$ is a constant density with the same mass as the difference between $f$ and its truncation $f\wedge p$. We can see that, as $p\to\infty$, we have $\mu_p\destar\mu$; hence $\lim_{p\to\infty} W_2^2(\mu_p,\nu)=W_2^2(\mu,\nu)$, and we also have $\limsup_{p\to\infty} \G[\mu_p]\leq \G[\mu]$.   
\end{proof}

As sometimes done already, in the sequel we will often identify absolutely continuous measures with their densities.

\subsection{Discrete maximum principle}

\begin{lem}\label{lem:V-gen}
	Given $g \in \M(\Omega)\cap L^1(\Omega)$ and a convex lower semicontinuous function $V\colon [0,\infty)\to \R\cup \{+\infty\}$ so that $V(0)=+\infty$, $\lim_{s \to +\infty}V(s) =+\infty$ and $V$ is of class $C^1$ on $(0,\infty)$, let $f_*$ minimise the functional
$$
	f \mapsto \int_\Omega  V\(\frac{f}{m}\) m +\frac{W_2^2(f,g)}{2\tau}
$$
on $L^1(\Omega)$. Then, there exists a constant $C>0$ such that
$$
V'\(\frac{f_*}{m}\)+\frac{\varphi}{\tau}=C\quad \text{ almost everywhere on } \Omega,
$$
where $\varphi$ is the unique (up to additive constants) Kantorovich potential from $f_*$ to $g$.
\end{lem}

For the readers' convenience, we recall that the definition and role of Kantorovich potentials. First, we recall the duality result introduced by Kantorovich, which reads, in the case of the quadratic cost $c(x,y)=|x-y|^2/2$
$$ \inf_{\pi \in \Pi(\mu,\nu)}  \int_{\Omega \times \Omega} \frac12 |x-y|^2 \di \pi(x,y) =\sup_{\varphi,\psi\,:\,\varphi(x)+\psi(y)\leq  \frac12 |x-y|^2}\int \varphi\di \mu+\int\psi\di\nu.$$
When $(\varphi,\psi)$ is an optimal pair in the above supremum, then we say that $\varphi$ is a Kantorovich potential from $\mu$ to $\nu$. The Kantorovich potential is always a Lipschitz (when $\Omega$ is compact) and semiconcave function, and is unique up to additive constants as soon as $\mu$ has strictly positive density almost everywhere. Moreover, it is connected to the optimal transport map $T$ via $T(x)=x-\nabla\varphi(x)$, for almost every $x \in\Omega$, and it also plays the role of first variation of the functional $\mu\mapsto \frac12 W_2^2(\mu,\nu)$. By inverting the roles of the two measures and using the uniqueness of the optimal map, it is easy to obtain $\nabla\psi=-\nabla\varphi\circ T$. We refer the reader to \cite[Sections 1.2, 1.3 and 7.2.2]{OTAM} for these facts and more details.

\begin{proof}
	From $V(0)=+\infty$ and from the finiteness of $\int_\Omega  V\(f_*/m\) m $ we deduce that $f>0$ almost everywhere. Hence, we can use \cite[Proposition 7.20]{OTAM} to deduce that $V'\(f_*/m\)+\varphi/\tau$ is equal to a constant almost everywhere on the support of $f$, i.e., the domain $\Omega$.
\end{proof}

\begin{lem}[Discrete maximum principle] \label{lem:max}
	Given $\nu \in \M(\Omega)$ with $\nu\ll \di x$, let $\mu_*$ be the unique minimiser for \eqref{minGFnu}. Then, for any $c_0, C_0>0$,
\begin{itemize}
\item if $\nu\leq C_0 m$, then $\mu_*$ is absolutely continuous with density $f_* \leq C_0m$;
\item if $\nu\geq c_0 m$, then $\mu_* \geq c_0m$.
\end{itemize}
 \end{lem}
 
\begin{proof}
	We prove the same estimates for the minimisation problem \eqref{eq:min-approx}, where $(\nu_\e)_\e$ is a smooth and strictly positive approximation of $\nu$ also satisfying the bound $\nu_\eps\leq C_0 m$ or the bound $\nu_\eps\geq c_0 m$ for all $\eps>0$. Then, by Lemma \ref{approxUeps}, as the constants $c_0$ and $C_0$ do not depend on $\eps$, the same estimates hold true for the minimiser of \eqref{minGFnu}. Also, we prove the result under the assumption that $m$ be Lipschitz continuous; then, a simple approximation argument gives the result for any $m$. 
	
	Let $\eps>0$ and let $\mu_\eps$ be the unique minimiser of Problem \eqref{eq:min-approx}. Let $\varphi$ be the Kantorovich potential from $\mu_\eps$ to $\nu_\e$. By Lemma \ref{approxUeps}, $\mu_\eps$ is absolutely continuous ($\mu_\eps = f_\eps \di x$), and, by Lemma \ref{lem:V-gen} applied with $V=U_\eps$, we obtain that
 $$
 U'_\eps\(\frac{f_\eps}{m}\)+\frac{\varphi}{\tau}=C \quad \text { almost everywhere on } \Omega.
 $$
Since $\varphi$ is at least Lipschitz continuous (see for instance \cite{OTAM}), this implies that $U'_\eps\(f_\eps/m\)$ is as well Lipschitz continuous. Using the explicit expression for $U_\eps$ and $U_\eps''\geq \eps>0$, we get that $f_\eps/m$ is Lipschitz continuous, and so the same is true of $f_\eps$ (by the assumption that $m$ is Lipschitz continuous). Moreover, $f_\eps$ is bounded from below by a positive constant since $U'_\eps\(f_\eps/m\)$ is bounded and $U'(0)=U_\eps'(0)=-\infty$. Since the target measure $\nu_\eps$ is supposed to be smooth and strictly positive, we face an optimal transport problem between two Lipschitz densities which are bounded below and either periodic (if $\Omega=\bb T^d$) or supported on a convex domain. In the former case we can apply the regularity result in \cite{Cor99} and in the latter case we can apply Caffarelli's regularity theory (see \cite{caf1,caf3,caf2}, \cite[Theorem 3.3]{DePFig survey} and \cite[Theorem 4.23 and Remark 4.25]{Figbook}) to get that $\varphi\in C^{2,\beta}(\Omega)$ for some $\beta<1$, under the extra assumption that $\Omega$ be uniformly convex and smooth, so that we have regularity of $T$ up to the boundary. Note that we get rid of the extra assumption on $\Omega$ at the end of the proof. Moreover, the optimal map $T={\rm id}-\nabla\varphi$ is a diffeomorphism and sends $\partial\Omega$ into $\partial\Omega$, which is only pertinent in the case where $\Omega$ is a convex bounded domain.
 \medskip
 
{\it - The case where $\Omega$ is the torus.} Let $\bar x$ be a point of maximum for $f_\eps/m$. Since $U_\eps'$ is monotonically increasing, then $\bar x$ is also a point of maximum for $U_\eps'\(f_\eps/m\)$. This implies that $\bar x$ is a point of minimum for $\varphi/\tau$.  Therefore, because $\Omega = \mathbb{T}^d$,
 $$
 \nabla \varphi(\bar x)=0, \quad D^2\varphi(\bar x)\ge 0.
 $$
 Let us recall that the optimal transport map $T\colon \Omega \to \Omega$ from $f_\eps$ to $\nu_\eps$ is given by
 $$
 T(x)=x-\nabla \varphi(x) \quad \mbox{for all }x \in \Omega.
 $$
From  $\nabla \varphi(\bar x)=0$ we obtain $T(\bar x)=\bar x$. Also, by the Monge--Amp\`ere equation,
 $$
\frac{f_\eps}{m}(\bar x)= \frac{\nu_\eps(T(\bar x))}{m(\bar x)}\det(\nabla T(\bar x))=\frac{\nu_\eps}{m}(T(\bar x))\det({\rm Id}-D^2\varphi(\bar x)).
 $$
 Since by assumption $\nu_\eps\le C_0m$ and we know $D^2\varphi(\bar x)\ge 0$, we get
 $$
\frac{f_\eps}{m}(\bar x)\le C_0\det({\rm Id}-D^2\varphi(\bar x)) \le C_0.
 $$
This proves the first part of the statement (i.e., the absolute continuity and the upper bound) for the case of the torus, and the second part (i.e., the lower bound) is analogous (choosing a minimum point for $f_\eps/m$ instead of a maximum point).
 \smallskip
 
{\it - The case where $\Omega$ is a uniformly convex, smooth and bounded domain.} The difficulties arise when $\bar x\in\partial\Omega$. To perform the same analysis as above we need either to exclude this case or to guarantee that anyway $\nabla \varphi(\bar x)=0$. 
 
{\it Step 1: upper bound.} If $\bar x$ is a minimum point for $\varphi$ and $\bar x\in\partial\Omega$, then $\nabla\varphi(\bar x)$ is orthogonal to the boundary, and $\nabla\varphi(\bar x)\cdot n(\bar x)\leq 0$, where we recall that $n(\bar x)$ denotes the outward unit normal vector to $\partial \Omega$ at $\bar x$. Yet, the strict convexity of $\Omega$ and the condition $T(\bar x)=\bar x-\nabla \varphi(\bar x)\in \bar\Omega$ impose $\nabla \varphi(\bar x)\cdot n(\bar x)> 0$, which is a contradiction. The upper bound is thus easily handled.
 
{\it Step 2: lower bound.} For the lower bound the situation is trickier, as the above contradiction does not work. Yet, we can use the fact that $T$ is a homeomorphism and that, for $\bar x \in \partial \Omega$ a maximum point for $\varphi$, $T(\bar x)$ must be a point of $\partial\Omega$ which, by monotonicity of $T$, must satisfy $n(T(\bar x))\cdot n(\bar x)\geq 0$. Indeed, setting $v=n(T(\bar x))$, in case  $n(\bar x)\cdot v< 0$, then we can find points $x'\in\Omega$ of the form $x'=\bar x+tv$ for $t>0$; from $T(x')\in\Omega\setminus\{T(\bar x)\}$ we deduce $(T(x')-T(\bar x))\cdot v< 0$, which contradicts the condition $(T(x')-T(\bar x))\cdot (x'-\bar x)\geq 0$ given by the monotonicity of $T$.
 
 Coming back to the point $y=T(\bar x)$, we can say that $y\in\partial\Omega$, $n(y)\cdot n(\bar x)\geq 0$, but also $y=\bar x-\nabla\varphi(\bar x)$. Moreover, since $\bar x$ is a maximum point for $\varphi$, we have $\nabla\varphi(\bar x)=tn(\bar x)$ for $t\geq 0$. Using $\bar x\in\Omega$ we have $(\bar x-y) \cdot n(y)\leq 0$, hence $n(y)\cdot \nabla  \varphi(\bar x)\leq 0$, which is a contradiction if $\nabla \varphi(\bar x)\neq 0$. Hence in this case we cannot exclude $\bar x\in\partial\Omega$, but we can guarantee that $\nabla \varphi(\bar x)=0$. This, together with the regularity of $\varphi$, allows to apply the second-order condition on $\varphi$ as in the torus case and conclude $f_\eps \geq c_0m$.\smallskip 
 
 {\it - The case where $\Omega$ is a convex bounded domain.} We now get rid of the extra assumption on $\Omega$. In this case the regularity theory of Caffarelli does not extend to the boundary, but the estimates can just be obtained by approximation, replacing $\Omega$ with a sequence of domains satisfying Caffarelli's assumptions, and then passing to the limit. The bounds being independent on the smoothness of the boundary and of its uniform convexity, the result stays true in general.
\end{proof}

\subsection{A priori $BV$ estimate}

We give here an a priori $BV$ estimate which we obtain also using the discrete maximum principle given in Lemma \ref{lem:max}.
 For convenience, we define a $BV$ norm weighted by $m$: if $u\in BV(\Omega)$, we set
$$
	\|u\|_{BV(\Omega;m)}:=\int_\Omega m \di |\nabla u|,
$$
where the right-hand side stands for the integral of the continuous function $m$ with respect to the scalar measure $|\nabla u|,$ which is the total variation of the vector measure $\nabla u$. Since we are always supposing that $m$ is bounded from above and below, this norm is bounded from above and below by constant multiples of the standard $BV$ norm. 

\begin{lem} \label{lem:BV}
Given $\nu \in \M(\Omega)$ with $\nu=g\di x$ for some $g \in L^1(\Omega)$, let $\mu_*$ be the unique minimiser of \eqref{minGFnu}. Assume that $c_0m\leq \nu\leq C_0 m$ for some $C_0,c_0>0$, so that by Lemma \ref{lem:max} we have $\mu_* = f_* \di x$ with $c_0m\leq f_* \leq C_0m$, and that $g/m \in BV(\Omega)$ and $D^2 (\log m) \leq \Lambda\, {\rm Id}$ for some $\Lambda \in \bb R$. Then there exists a constant $C_1>0$, depending on $c_0, C_0$ and on the sign of $\Lambda$, such that, if $C_1\Lambda\tau<1$, we have
$$
	\Bigl\| \frac{f_*}{m}\Bigr\|_{BV(\Omega;m)}\leq \frac{1}{1-C_1\Lambda \tau} \Bigl\| \frac{g}{m}\Bigr\|_{BV(\Omega;m)}.
 $$
 \end{lem}
\begin{proof}
Again, let us assume that $m$ is Lipschitz continuous; then a simple approximation argument gives the result for any $m$. We start, as usual, by writing the optimality conditions of \eqref{minGFnu}:
$$
	0=\nabla\(U'\(\frac{f_*}{m}\)+\frac{\varphi}{\tau}\)=U''\(\frac{f_*}{m}\)\nabla \(\frac{f_*}{m}\) +\frac{\nabla \varphi}{\tau} \quad \mbox{almost everywhere on $\Omega$},
$$
where $\varphi$ is the Kantorovich potential from $f_*$ to $g$. Note that the optimality conditions themselves imply that $f_*$ is a Lipschitz function, which allows us to differentiate it almost everywhere. We now write, for any vector $v\in \R^d\setminus \{0\},$ $\widehat{v} := v/ |v|$, and set, by convention, $\widehat{0}=0$.
Since 
$$
U''\(\frac{f_*}{m}\)\nabla \(\frac{f_*}{m}\)=-\frac{\nabla \varphi}{\tau}
$$
it follows, since $U$ is convex, and hence $U''\geq 0$, that we have
$$
\bigg|\nabla \(\frac{f_*}{m}\)\bigg|=\nabla \(\frac{f_*}{m}\)\cdot \(-\widehat{\nabla \varphi}\).
$$
Hence,
$$
\int_\Omega m\bigg|\nabla \(\frac{f_*}{m}\)\bigg|=\int_\Omega m \nabla \(\frac{f_*}{m}\)\cdot \(-\widehat{\nabla \varphi}\)
=\int_\Omega \nabla f_*\cdot \(-\widehat{\nabla \varphi}\)+\int_\Omega f_* \frac{\nabla m}{m}\cdot \widehat{\nabla \varphi}=:I_1+I_2.
$$
 By \cite[Lemma 3.1]{DePMesSan} applied with $H=\left|\cdot\right|$ (note that one should first write the inequality below for $H_\eps=\sqrt{\eps^2+|\cdot|^2}$, $\eps>0$, instead of $H$ and then pass to the limit $\eps \to 0$, which explains the choice of the convention for $\widehat 0$; also, one should first approximate $g$ in $W^{1,1}(\Omega)$ and then pass to the limit at the very end of the proof, since otherwise the integral here below is not well-defined), we get
 $$
 I_1\leq \int_\Omega \nabla g\cdot \widehat{\nabla \psi},
 $$
 where $\psi$ is the Kantorovich potential from $g$ to $f$, and $\widehat{\nabla \psi} := \nabla \psi/ |\nabla \psi|$.
 Also, since the optimal map from $g$ to $f$ is $S={\rm id} -\nabla\psi$,
 $$
I_2= \int_\Omega (S_\#g) \frac{\nabla m}{m}\cdot \widehat{\nabla \varphi}=
\int_\Omega \left(\frac{\nabla m}{m}\circ S\right) \cdot \left(\widehat{\nabla \varphi}\circ S\right)g.
 $$
Since $\widehat{\nabla \varphi}\circ S=-\widehat{\nabla \psi},$ this gives
$$
I_2=-\int_\Omega \( \frac{\nabla m}{m}\circ S \) \cdot \widehat{\nabla \psi}\,g.
$$
 Hence,
 \be\label{eq:fstar}
 \begin{split}
 \int_\Omega m\bigg|\nabla \(\frac{f_*}{m}\)\bigg|&\leq \int_\Omega \(\nabla g-g \frac{\nabla m}{m}\circ S\)\cdot \widehat{\nabla \psi}\\
 &= \int_\Omega \(\nabla g-g \frac{\nabla m}{m}\)\cdot \widehat{\nabla \psi}+ \int_\Omega g\(\frac{\nabla m}{m}- \frac{\nabla m}{m}\circ S\)\cdot \widehat{\nabla \psi}\\
 &\leq \int_\Omega \bigg|\nabla g-g \frac{\nabla m}{m}\biggr| 
 + \int_\Omega g\(\frac{\nabla m}{m}- \frac{\nabla m}{m}\circ S\)\cdot \widehat{\nabla \psi}\\
  &= \int_\Omega m\bigg|\nabla \(\frac{g}{m}\)\bigg|
 + \int_\Omega g\(\frac{\nabla m}{m}- \frac{\nabla m}{m}\circ S\)\cdot \widehat{\nabla \psi}.
 \end{split}
 \ee
 Since $\frac{\nabla m}{m}=\nabla \log m$ and $D^2 (\log m) \leq \Lambda\, {\rm Id}$, then, noting that we have
 $$
 	y-S(y)=\nabla\psi(y)=\widehat{\nabla \psi}|\nabla \psi(y)|=\widehat{\nabla \psi}|y-S(y)|,
$$
 we also get
 $$
\left( \frac{\nabla m}{m}(y)- \frac{\nabla m}{m}\circ S(y)\right) \cdot \widehat{\nabla \psi}
 \leq \Lambda |y-S(y)|=\Lambda |\nabla \psi(y)|.
 $$
 Thus,
 $$
  \int_\Omega g\(\frac{\nabla m}{m}- \frac{\nabla m}{m}\circ S\)\cdot \widehat{\nabla \psi}
  \leq \Lambda \int_\Omega |\nabla \psi| g. 
 $$
 Note that, since $\nabla \psi\circ T=-\nabla \varphi$ ($T$ being the optimal map from $f_*$ to $g$),
 $$
  \int_\Omega |\nabla \psi| g= \int_\Omega |\nabla \psi| (T_\#f_*)=
 \int_\Omega (|\nabla \psi|\circ T) f_*= \int_\Omega |\nabla \varphi| f_* = \tau \int_\Omega m \frac{f_*}{m}\bigg|\nabla U'\(\frac{f_*}{m}\)\bigg|.$$
 The last term can be rewritten using 
$$\frac{f_*}{m}\bigg|\nabla U'\(\frac{f_*}{m}\)\bigg|= \frac{f_*}{m}U''\(\frac{f_*}{m}\)\bigg|\nabla \(\frac{f_*}{m}\)\bigg|.
 $$
Also using the fact that the function $s\mapsto sU''(s)$ is nonincreasing, we can go on with the estimates: if $\Lambda>0$, then we have
$$
  \int_\Omega g\(\frac{\nabla m}{m}- \frac{\nabla m}{m}\circ S\)\cdot \widehat{\nabla \psi}
  \leq \Lambda \int_\Omega |\nabla \psi| g \leq \Lambda c_0U''(c_0)\tau \int_\Omega m \bigg|\nabla \(\frac{f_*}{m}\)\bigg|;
 $$
if $\Lambda<0$, then we obtain
$$
  \int_\Omega g\(\frac{\nabla m}{m}- \frac{\nabla m}{m}\circ S\)\cdot \widehat{\nabla \psi}
  \leq \Lambda \int_\Omega |\nabla \psi| g \leq \Lambda C_0U''(C_0)\tau \int_\Omega m \bigg|\nabla \(\frac{f_*}{m}\)\bigg|.
 $$
 
 In all cases, \eqref{eq:fstar} yields
 $$
 \int_\Omega m\bigg|\nabla \(\frac{f_*}{m}\)\bigg|\leq \int_\Omega m\bigg|\nabla \(\frac{g}{m}\)\bigg|
 +C_1\Lambda \tau  \int_\Omega m \bigg|\nabla \(\frac{f_*}{m}\)\bigg|, 
 $$
 where $C_1>0$ depends on $c_0,C_0$ and on the sign of $\Lambda$. This means, provided $C_1\Lambda\tau<1$,
 $$
  \int_\Omega m\bigg|\nabla \(\frac{f_*}{m}\)\bigg|\leq \frac{1}{1-C_1\Lambda \tau} \int_\Omega m\bigg|\nabla \(\frac{g}{m}\)\bigg|,
 $$
 which is the desired result.
 \end{proof}

 \begin{remark}
If $\rho\equiv 1$, then the previous result holds with $\Lambda=0$, i.e., 
\begin{equation*}
	\left\|\left( \frac{f_*}{m}\right) \right\|_{BV(\Omega;m)} \leq \left\| \left( \frac{g}{m}\right) \right\|_{BV(\Omega;m)} .\end{equation*}
\end{remark}

\subsection{A priori $H^1$ estimates}\label{subsec:H1}
These estimates will be needed to prove later that the solution of the JKO gives a weak solution. They are obtained by the flow-interchange technique \cite{MattMcCSav} after proving the following result:
%

 \begin{lem}\label{lem:geodesic-conv}
Let $q>1$. If $D^2(\log m)\leq \Lambda\, {\rm Id}$ for some $\Lambda>0$ and $C>0$, then the functional 
 $$
 	f\mapsto \G_{(q)}[f]:=\int_\Omega \left(\frac f m\right)^qm
$$
 is $\tilde\Lambda$-geodesically convex for $\tilde\Lambda=-\Lambda(q-1)\left(\frac{C}{\inf m}\right)^{q-1}$ on the set of densities $f$ with $\|f\|_{L^\infty}\leq C$.  If $D^2(\log m)\leq 0$ (i.e., $\Lambda=0$), then this same functional is geodesically convex without any $L^\infty$ restriction.
\end{lem}
\begin{proof} First, we note that the set of densities satisfying a given $L^\infty$ upper bound is geodesically convex in the Wassterstein space. The proof follows the same scheme as the usual one when no spatial inhomogeneity $m$ is present; see \cite{MC} and, for instance, \cite[Chapter 7]{OTAM}.

Given two densities $f_0$ and $f_1$, we know that the density $f_\alpha$ of the Wasserstein geodesic connecting $f_0$ and $f_1$ is given, for all $\alpha\in[0,1]$ and $x \in \Omega$, by 
$$
	f_\alpha(x)=\frac{f_0}{\det(DT_\alpha)}((T_\alpha)^{-1}(x)),\quad\mbox{ where }T_\alpha=(1-\alpha){\rm id}+\alpha T,
$$
and $T$ is the optimal transport map from $f_0$ to $f_1$. With the change of variable $x=T_\alpha(y)$ we get 
$$ 
	\int_\Omega \left(\frac{f_\alpha}{m}\right)^qm=\int_\Omega f_0(y)^q \exp\left((1-q)(a(t,y)+b(t,y))\right)\di y,
$$
where, for all $(\alpha,x) \in [0,1]\times \Omega$,
$$
a(\alpha,y)=\log(\det(DT_\alpha(y))),\quad b(\alpha,y)=\log (m(T_\alpha(y))).
$$
We now differentiate twice in $\alpha$, and use that $(\exp(h(\alpha)))''=\exp(h(\alpha))[(h'(\alpha))^2+h''(\alpha)]\geq \exp(h(\alpha))h''(\alpha)$ for any twice differentiable function $h\colon [0,1] \to \R$. Moreover, standard results on the concavity properties of the determinant of positive-definite matrices imply $a''\leq 0$, and our assumption on $m$ implies $b''(\alpha,y)\leq \Lambda|y-T(y)|^2$ for all $\alpha \in [0,1]$. We then get
$$
	(\G_{(q)}[f_\alpha])''\geq -(q-1)\Lambda \int_\Omega f_0(y)^q \exp\left((1-q)(a(\alpha,y)+b(\alpha,y))\right)|y-T(y)|^2\di y.
$$
In the case $\Lambda=0$ we stop here and we obtain geodesic convexity of $\G_{(q)}$. Otherwise, we go on by rewriting the exponential and we have
$$
	(\G_{(q)}[f_\alpha])''\geq-(q-1)\Lambda\int_\Omega f_0(y)|y-T(y)|^2\left(\frac{f_0(y)}{\det(DT_\alpha(y))m(T_\alpha(y))}\right)^{q-1}\di y.
$$
Assuming $f_0,f_1\leq C$ implies $f_\alpha\leq C$ for all $\alpha \in [0,1]$. Using that $f_0/\mathrm{det}(DT_\alpha)=f_\alpha\circ T_\alpha$ and $m\circ T_\alpha\geq \inf m$ we obtain
$$
	(\G_{(q)}[f_\alpha])''\geq-(q-1)\Lambda \left(\frac{C}{\inf m}\right)^{q-1}\int_\Omega  f_0(y)|y-T(y)|^2dy=-(q-1)\Lambda \left(\frac{C}{\inf m}\right)^{q-1}W_2^2(f_0,f_1),
$$
which is exactly the claim.
\end{proof}

\begin{lem}\label{H1estimate}
	Given $\nu \in \M(\Omega)$ with $\nu=g\di x$ for some $g \in L^1(\Omega)$, let $\mu_*$ be the unique minimiser of \eqref{minGFnu}. Assume that $\nu\leq C_0 m$, so that by Lemma \ref{lem:max} we have $\mu_* = f_* \di x$ with $f_* \leq C_0m$. Then, for all $q>1$, we have
$$\tau c(r,q) \int_\Omega m \left|\nabla\left(\frac{f_*}{m}\right)^{\frac{q-(r+1)}{2}}\right|^2\leq \G_{(q)}[g]-\G_{(q)}[f_*]+(q-1)\Lambda \left(\frac{C_0\sup m}{\inf m}\right)^{q-1}W_2^2(f_*,g),$$
for a constant $c(r,q)>0$. 
\end{lem}
 
\begin{proof}
The proof is based on the so-called {\it flow-interchange} procedure, first introduced in \cite{MattMcCSav}; however, we will follow the technique described in \cite{LavSan}. First, we  write the optimality conditions for the minimisers of  \eqref{minGFnu}: we have, almost everywhere in $\Omega$,
$$
	\tau \nabla \(U'\(\frac{f_*}{m}\)\)=-\nabla\varphi.
$$
We then multiply this equality by $f\nabla (V'(f_*/m))$, for a convex function $V$, and integrate. This provides
$$
	\tau\int_\Omega m \frac{f_*}{m} U''\(\frac{f_*}{m}\)V''\(\frac{f_*}{m}\) \left|\nabla\left(\frac{f_*}{m}\right)\right|^2=-\int_\Omega f_*\nabla \left(V'\left(\frac{f_*}{m}\right)\right)\cdot\nabla\varphi.
$$ 
Note that the right-hand side corresponds to the derivative, computed at time $\alpha=0$, of $\alpha\mapsto \int_\Omega mV(f_\alpha/m)$, where $f_\alpha$ is the Wasserstein geodesic from $f_0=f_*$ to $f_1=g$. Choosing $V(s)=s^q$, and using $U(s)=s^{-r}$, this provides
$$
	\tau c(r,q) \int_\Omega m \left|\nabla\left(\frac{f_*}{m}\right)^{\frac{q-r-1}{2}}\right|^2=\frac{\di}{\di s}\G_{(q)}[f_\alpha]_{|\alpha=0}.
$$ 

The claim is then proved by using Lemma \ref{lem:geodesic-conv}: indeed, setting $h(\alpha)=\G_{(q)}[f_\alpha]$, we have $h''\geq -(q-1)\Lambda \left(\frac{C_0\sup m}{\inf m}\right)^{q-1}$ and $h(1)\geq h(0)+h'(0)+\frac 12 \inf_{\alpha\in [0,1]} h''(\alpha)$, which yields the result.
\end{proof}
   
\begin{remark}
	If $\rho\equiv 1$, then we fall into the assumptions of Lemmas \ref{lem:geodesic-conv} and \ref{H1estimate} with $\Lambda=\tilde \Lambda=0$.
\end{remark}

The following lemma is a classical fact of the JKO scheme, and luckily does not use geodesic convexity (indeed, negative power functionals are rarely geodesically convex). We state it for completeness, but the reader can see that it is possible to obtain the desired estimates of this paper without using it.

\begin{lem}\label{H1estimate-JKO}
	Given $\nu \in \M(\Omega)$ with $\nu=g\di x$ for some $g \in L^1(\Omega)$, let $\mu_*$ be the unique minimiser of \eqref{minGFnu}. Assume that $\nu\leq C_0 m$, so that by Lemma \ref{lem:max} we have $\mu_* = f_* \di x$ with $f_* \leq C_0m$. Then, we have
$$\tau \frac{r^2}{(r+1/2)^2} \int_\Omega m \left|\nabla\left(\frac{f_*}{m}\right)^{-r-1/2}\right|^2=\frac{W_2^2(f_*,g)}{\tau}\leq 2\left(\mathcal F_\rho[g]-\mathcal F_\rho[f_*]\right).$$
\end{lem}
 
\begin{proof}
	The computations come again from the optimality conditions $\tau \nabla (U'(\frac{f_*}{m}))=-\nabla\varphi$. We square and integrate with respect to $f$, thus obtaining
$$\tau  \int_\Omega f \left|\nabla U'\left(\frac{f_*}{m}\right)\right|^2=\frac 1\tau\int_\Omega f|\nabla \varphi|^2=\frac{W_2^2(f_*,g)}{\tau}.$$
We then compute, almost everywhere in $\Omega$,
$$
	f \left|\nabla U'\left(\frac{f_*}{m}\right)\right|^2=m\frac{f_*}{m}(r+1)^2 \left(\frac{f_*}{m}\right)^{-2(r+1)}\left|\nabla\left(\frac{f_*}{m}\right)\right|^2= \frac{r^2}{(r+1/2)^2}m \left|\nabla\left(\frac{f_*}{m}\right)^{-r-1/2}\right|^2,
$$
and obtain the equality in the claim. In order to compare the Wasserstein distance to the functional $\mathcal F_\rho$, it is enough to use the optimality of $f_*$, i.e.,
$$
	\mathcal F_\rho[f_*]+\frac{W_2^2(f,g)}{2\tau}\leq \mathcal F_\rho[g],
$$
which gives the inequality in the claim.
\end{proof}

\section{Existence and uniqueness: proof of Theorem \ref{thm:existence-uniqueness}}\label{sec:existence-uniqueness}

\subsection{Preliminary on the notion of weak solution}

In order to clarify notation, let us 
\begin{itemize}
\item define the homogeneous Sobolev space
 $$
 	\dot H^1(\Omega):=\{\varphi\colon\Omega\to \R \ \  \mbox{s.t.}\ \ \nabla \varphi \in L^2(\Omega)\}
$$
 endowed with the norm $\|\nabla \cdot\|_{L^2(\Omega)}$;
 \item write $H^{-1}$ for the dual of $H^1$ (and not of $H^1_0$): 
  $$
	H^{-1}(\Omega):=\left(H^1(\Omega)\right)'.
$$
 \end{itemize}
 
We want now to prove some preliminary lemmas that will be useful throughout. For the first one, we recall our standing assumption on $m$ as it really is crucial here.
 \begin{lem} \label{lem:W1infty}
	Assume that $\log m \in W^{1,p}(\Omega)$ for some $p>d$. If $u$ is a weak solution of \eqref{eq:u}, then 
$$
	\partial_t u \in L^2_{\rm loc}([0,\infty),H^{-1}(\Omega))
$$
and the norm of $\partial_t u$  in $L^2([0,\infty),H^{-1}(\Omega))$ only depends on the norm of  $\log m$ in $W^{1,p}(\Omega)$ and on the norm of $u^{-r}$ in $L^2([0,\infty),\dot H^{1}(\Omega))$.

 \end{lem}
 \begin{proof}
%
 
Recalling \eqref{eq:u} we get:
 $$
\partial_t u \in L^2_{\rm loc}([0,\infty),H^{-1}(\Omega))\qquad \Leftrightarrow\qquad 
\frac{1}{m(x)}\div_x\bigl(m(x)\nabla_x (u(t,x)^{-r})\bigr) \in L^2_{\rm loc}([0,\infty),H^{-1}(\Omega)).
 $$
Let $T>0$. By definition,
 \begin{align*}
	&\left\|\frac{1}{m}\div_x\bigl(m\nabla_x (u^{-r})\bigr)\right\|_{L^2([0,T],H^{-1}(\Omega))}\\
&\phantom{=}=\sup_{\substack{\varphi \in L^2([0,T],H^{1}(\Omega))\\ \left\|\varphi\right\|_{L^2([0,T],H^1(\Omega))} = 1} }\int_0^T \int_\Omega  \frac{\varphi}{m}(t,x) \div_x\bigl(m(x)\nabla_x u^{-r}(t,x)\bigr) \di x \di t\\
&\phantom{=}=\sup_{\substack{\varphi \in L^2([0,T],H^{1}(\Omega))\\ \left\|\varphi\right\|_{L^2([0,T],H^1(\Omega))} = 1} } \Bigg( -\int_0^T \int_\Omega  \nabla_x \varphi(t,x) \cdot \nabla_x u^{-r}(t,x) \di x \di t\\
&\phantom{==} +\int_0^T \int_\Omega  \varphi(t,x) \nabla_x \left( \log m(x) \right) \cdot \nabla_x \bigl(u^{-r}(t,x)\bigr) \di x \di t \Bigg).
\end{align*}
For the first term, since by the definition of weak solution $u^{-r}\in L^2([0,T],H^1(\Omega))$,
by H\"older's inequality we get
$$
\biggl|\int_0^T \int_\Omega  \nabla_x \varphi(t,x) \cdot \nabla_x u^{-r}(t,x)\biggr| \di t \di x
\leq 
\bigl\|\nabla_x\varphi\bigr\|_{L^2([0,T],L^2(\Omega))}
\bigl\|\nabla_x(u^{-r})\bigr\|_{L^2([0,T],L^2(\Omega))}.
$$
Hence,
$$
\sup_{\substack{\varphi \in L^2([0,T],H^{1}(\Omega))\\ \left\|\varphi\right\|_{L^2([0,T],H^1(\Omega))} = 1} }\biggl|\int_0^T \int_\Omega  \nabla_x \varphi(t,x) \cdot \nabla_x u^{-r}(t,x) \di x \di t\biggr|<\infty.
$$
For the second term, note that, when $d> 2$, Sobolev's inequality yields $\varphi \in L^2([0,T],L^{2d/(d-2)}(\Omega))$ and therefore, thanks to the assumption that $\log m \in W^{1,p}(\Omega)$ with $p>d$ and to H\"older's inequality,
\begin{align*}
	&\biggl|\int_0^T \int_\Omega  \varphi(t,x)\nabla(\log m(x))\cdot \nabla_x (u^{-r}(t,x)) \di x \di t \biggr| \\
	&\phantom{=} \leq \bigl\|\varphi\nabla(\log m)\bigr\|_{L^2([0,T],L^2(\Omega))} \bigl\|\nabla_x(u^{-r})\bigr\|_{L^2([0,T],L^2(\Omega))}\\
	&\phantom{=}\leq  \bigl\|\nabla(\log m)\bigr\|_{L^d(\Omega)} \bigl\|\varphi\bigr\|_{L^2([0,T],L^{2d/(d-2)}(\Omega))} \bigl\|\nabla_x(u^{-r})\bigr\|_{L^2([0,T],L^2(\Omega))},
\end{align*}
so that we conclude
$$
\sup_{\substack{\varphi \in L^2([0,T],H^{1}(\Omega))\\ \left\|\varphi\right\|_{L^2([0,T],H^1(\Omega))} = 1} }\biggl|\int_0^T \int_\Omega   \varphi(t,x)\nabla(\log m(x))\cdot \nabla_x \left( u^{-r}(t,x)\right) \di x \di t\biggr|
<\infty.
$$
When $d\leq 2$, instead of $\varphi \in L^2([0,\! T],\!L^{2d/(d-2)}\!(\Omega)\!)$, Sobolev's inequality provides $\varphi \in L^2([0,\!T],\!L^{q}\!(\Omega)\!)$ for every $q<\infty$ and we can still conclude thanks to the assumption $\log m \in W^{1,p}(\Omega)$ with the \emph{strict} inequality $p>d$. 
 \end{proof}
%
%

  \begin{lem} \label{lem:holder}
Let $\alpha, \beta \in \R$ be distinct. Assume that both $u^\alpha$ and $u^\beta$ belong to $L^2([0,\infty),\dot  H^1(\Omega))$. Then $u^\gamma\in L^2([0,\infty),\dot H^1(\Omega))$ for all $\gamma \in (\alpha,\beta)$, and the norm of $u^\gamma$ can be estimated by those of $u^\alpha$ and $u^\beta$.
 \end{lem}
 
 \begin{proof}
 Note that, for all $\eta \in \R$:
 \begin{align*}
 u^\eta \in L^2([0,\infty),\dot  H^1(\Omega))&\quad \Leftrightarrow \quad \int_0^\infty\int_\Omega|\nabla_x u^\eta(t,x)|^2\di x \di t<\infty\\
 &\quad \Leftrightarrow \quad \int_0^\infty\int_\Omega u^{2(\eta-1)}(t,x)|\nabla_x u|^2\di x \di t<\infty.
\end{align*}
 Hence, by H\"older's inequality, if $\gamma\in (\alpha,\beta)$ we have
 \begin{align*}
 &\int_0^\infty\int_\Omega u^{2(\gamma-1)}(t,x) |\nabla_x u|^2(t,x) \di x\di t\\
 &\phantom{=}= \int_0^\infty\int_\Omega u^{2(\alpha-1)\frac{\gamma-\beta}{\alpha-\beta}} (t,x) |\nabla_x u|^{\frac{2(\gamma-\beta)}{\alpha-\beta}}(t,x)\, u^{2(\beta-1)\frac{\alpha-\gamma}{\alpha-\beta}} (t,x) |\nabla_x u|^{\frac{2(\alpha-\gamma)}{\alpha-\beta}}(t,x) \di x\di t\\
 &\phantom{=} \leq \biggl(\int_0^\infty\int_\Omega u^{2(\alpha-1)}(t,x)|\nabla_x u|^2(t,x)\di x \di t\biggr)^{\frac{\gamma-\beta}{\alpha-\beta}}
 \biggl(\int_0^\infty\int_\Omega u^{2(\beta-1)}(t,x)|\nabla_x u|^2(t,x)\di x \di t\biggr)^{\frac{\alpha-\gamma}{\alpha-\beta}}\\
 &\phantom{=}<\infty,
 \end{align*}
which implies the result.
 \end{proof}
 
 We will also need several times in the sequel the following stability result for weak solutions:
 
\begin{lem} \label{lem:stability}
	Suppose that $(f_n)_n$ is a sequence of solutions of \eqref{eq:f} associated with a sequence of weights $(m_n)_n$. Suppose that, for each $n$, $f_n$ is bounded both from above and below by positive constants which are not necessarily uniform in $n$, and suppose that the masses $M_n:=\int_\Omega f_n$ (preserved in time) tend to a value $M>0$ as $n\to\infty$. Suppose that $(\log m_n)_n$ is bounded in $W^{1,p}(\Omega)$ (for some $p>d$), that $\log m_n\to \log m$ uniformly as $n\to\infty$ for some $m$ with $\log m\in W^{1,p}(\Omega)$, that $(f_n(0))_n$ is bounded in $L^{r+3}(\Omega)$, that $(f_n(0)^{-1})_n$ is bounded in $L^r(\Omega)$, and that $f_n(0)\destar f_0$ as $n\to\infty$ for some $f_0 \in L^{r+3}(\Omega)$. Then the curves $(t\mapsto f_n(t))_n$ are equicontinuous as curves valued in $W_2(\Omega)$ and, up to a subsequence, $f_n(t)\destar f(t)$ as $n\to\infty$ for every $t\geq0$, where $f$ is a weak solution of \eqref{eq:f} starting from $f_0$ and associated with the weight $m$.   \end{lem}
 
 \begin{proof}
  For each $n$, set $u_n=f_n/m$; we can use the fact that $u_n^{\alpha}$ belongs to $L^2([0,\infty),\dot H^1(\Omega))$ for every $\alpha\in\R$ (the definition of weak solution guaranteeing this fact for $\alpha=1$, and the upper and lower bounds on $u_n$ allowing us to use in fact any $\alpha\in\R$) together with $\partial_t u_n \in L^2([0,\infty), H^{-1}(\Omega))$ (by the Sobolev behaviour of $\log m_n$; see Lemma \ref{lem:W1infty}) in order to compute, for $q\neq r+1$,
  \begin{equation}\label{H1estimate-semicont2}
  \frac{\di}{\di t} \G_{(q;m_n)}[f_n(t)]=-\frac{4q(q-1)r(r+1)}{(q-r-1)^2}\int_\Omega m_n |\nabla((u_n)^{(q-r-1)/2})|^2.
  \end{equation}
(Note that we detail a similar computation later in \eqref{eq:der Lq}.)
  
Let us first obtain a uniform bound on the $L^2_{\rm loc}([0,\infty),H^1(\Omega))$ norm of $u_n$ and $u_n^{-r}$. Using $q=r+3$ in \eqref{H1estimate-semicont2} one obtains the bound on the norm of $u_n$ in $L^2([0,\infty),\dot H^1(\Omega))$ in terms of $\G_{(r+3;m_n)}[f_n(0)]$, which is bounded by assumption. In order to transform this bound into a bound in $L^2_{\rm loc}([0,\infty), H^1(\Omega))$ we use the fact that the average of $u_n$ is bounded since $u_n> 0$ and $\int_\Omega u_n m_n = \int_\Omega f_n=M_n$. Using the Poincar\'e--Wirtinger inequality we also bound the $L^2$ norm, and we get the desired bound.
  
  Using $q=-r$ in \eqref{H1estimate-semicont2} one obtains the bound on the norm of $u_n^{-(r+1/2)}$ in $L^2([0,\infty),\dot H^1(\Omega))$ in terms of $\G_{(-r;m_n)}[f_n(0)]$, which is also bounded. Notice that we also obtain boundedness of the $L^r$ norm of $f_n(t)^{-1}$. Then, by Lemma \ref{lem:holder} we deduce that $u_n^{-r}$ is bounded in $L^2([0,\infty),\dot H^1(\Omega))$. The bound on $\int_\Omega f_n(t)^{-r}$ also provides a bound on the average of $u_n^{-r}$  and, again, using the Poincar\'e--Wirtinger inequality the bound becomes a bound in $L^2_{\rm loc}([0,\infty), H^1(\Omega))$.
  
 Each $f_n$ represents a continuous curve valued in the compact space $W_2(\Omega)$. In order to prove that these curves are equicontinous we recall the following fact from optimal transport theory (see for instance \cite[Chapter 5]{OTAM}): whenever a curve of positive measures $(\mu(t))_t$ with fixed mass  on $\Omega$ satisfies $\partial_t\mu+\div_x (\mu v)=0$ (with no-flux boundary conditions), then we have 
 $$\int_0^T|\mu'(t)|^2\di t\leq \int_0^T\int_\Omega |v(t,x)|^2\di\mu(t)\di t,$$ 
 where $|\mu'(t)|$ is the metric derivative (see \cite{AGS}) of $\mu$. It is an important fact that bounds on the $L^2$ norm of the metric derivative imply H\"older continuity, from standard Sobolev injections.  In our case, for $\mu=f_n$, using \eqref{eq:f}, the vector field $v$ is given by
 $\nabla (u_n^{-(r+1)})$ and estimating its $L^2$ norm exactly amounts to the estimate of $u_n^{-(r+1/2)}$ in $L^2([0,\infty),\dot H^1(\Omega))$. 
 
 We can therefore extract a uniformly converging subsequence (uniformly for the $W_2$ metric). In particular, up to a subsequence, we have a weak limit $f_n(t)\destar f(t)$ for every $t$.
  Moreover, the $L^2([0,\infty), H^1(\Omega))$ bound on $u_n$, together with the $L^2([0,\infty), H^{-1}(\Omega))$ bound on $\partial_t u_n$ (which comes from Lemma \ref{lem:W1infty} and the uniform Sobolev bound on $\log m_n$), allow us to apply the Aubin--Lions lemma (see \cite{Aubin}) and obtain strong compactness on $u_n$. This means that we can assume that we have strong and almost-everywhere convergence $f_n(t)\to f(t)$ as well as $u_n(t)\to u(t)$ and $u_n^{-r}(t)\to u^{-r}(t)$. Because of our bounds, we also have weak $L^2$ convergence of $\nabla (u_n^{-r})$ to $\nabla (u^{-r})$. Together with the strong $L^2$ convergence $m_n\to m$, this allows the equation satisfied by $f_n$ to pass to the limit, which gives the existence of a weak solution satisfying the desired bounds and with initial datum $f_0$ and weight $m$.
 \end{proof}

\subsection{Existence}\label{subsec:existence}
We first prove the existence part of the theorem. The key technical tool is the use of the JKO scheme, as developed in Section \ref{sec:time-disc}: given $f_0$ with $\int_\Omega f_0=M$ and a time-step $\tau>0$, one can define a recursive sequence via
\begin{equation}\label{recJKO}
\mu^{(\tau)}_{k+1}=\argmin\left\{\mathcal F_\rho[\mu]+\frac{W_2^2(\mu,\mu^{(\tau)}_{k})}{2\tau},\;\mu\in\mathcal M_M(\Omega)\right\},
\end{equation}
and we define $f^{(\tau)}_{k}$ as the density of $\mu^{(\tau)}_{k}$, for every $k\in \mathbb N$. From the estimates in Section \ref{sec:time-disc} (in particular Lemmas \ref{lem:max}, \ref{lem:BV} and \ref{H1estimate}) we know the following facts:
\begin{itemize}
	\item if for some $k\in\mathbb N$ we have $c_0m\leq f^{(\tau)}_{k}\leq C_0m$ for some $c_0,C_0$, then the same inequality stays true for $f^{(\tau)}_{k+1}$ (i.e. lower and upper bounds are preserved along the evolution);
	\item under the assumption that $D^2(\log m)\leq \Lambda\, {\rm Id}$ for some $\Lambda \in \br$, if for any $k\in \mathbb N$ we have $c_0m\leq f^{(\tau)}_{k}\leq C_0m$ and $f^{(\tau)}_{k}/m\in BV(\Omega)$, then also $f^{(\tau)}_{k+1}/m\in BV(\Omega)$ and we have 
$$
	\|f^{(\tau)}_{k+1}/m\|_{BV(\Omega;m)}\leq \frac{1}{1-C_1\Lambda\tau} \|f^{(\tau)}_{k}/m\|_{BV(\Omega;m)},
$$
for a positive constant $C_1$ (depending on $c_0,C_0$ and on the sign of $\Lambda$); this means in particular that $BV$ norms do not grow ``too'' fast during the evolution, provided we assume $L^\infty$ bounds on $f$ and semiconcavity of $\log m$;
	\item  possibly assuming a priori $L^\infty$ bounds on $f^{(\tau)}_k$ for every $k$ and semiconcavity of $\log m$, the $H^1$ norm of quantities of the form $\(f^{(\tau)}_{k}/m\)^p$ (for $p=(q-r-1)/2$, $q>1$, and $p=-(r+1/2)$) can be estimated by terms which are the addends of a telescopic sum in $k$ (allowing us to sum them and obtain integral estimates in time).
\end{itemize}

%

We first prove a more restrictive existence result. Indeed, it assumes extra semiconcavity of $\log m$ and boundedness and $BV$ regularity of the initial datum.

\begin{lem}\label{thm:existence-weak-1}
	Assume $D^2 (\log m) \leq \Lambda\, {\rm Id}$ for some $\Lambda\in\R$. Then, for any $f_0\in L^{r+3}(\Omega)$ with $\f_\rho[f_0]<\infty$ and $c_0,C_0$ satisfying $c_0m\leq f_0\leq C_0m$ and $f_0/m \in BV(\Omega)$, there exists a distributional solution of \eqref{eq:f}, starting from $f_0$, obtained as the limit of the JKO scheme.
Also, this solution satisfies
$$
\frac{f}{m} \in L^2([0,\infty),\dot H^1(\Omega))\cap L^2_{\rm loc}([0,\infty),H^1(\Omega)),\quad 
\Bigl(\frac{f}{m}\Bigr)^{-r} \in L^2([0,\infty),\dot H^1(\Omega))\cap L^2_{\rm loc}([0,\infty),H^1(\Omega)),
$$
and is therefore a weak solution according to Definition \ref{def:weak}.
\end{lem}

\begin{proof}
The proof follows the scheme described in \cite[Chapter 8]{OTAM} to prove the convergence of the JKO iterations. Following such a scheme (see also \cite{DiMMauSan}, where this general procedure is presented), one has a sequence $(f_k)_k$ of densities obtained by iteratively solving the minimisation problem \eqref{recJKO}. For simplicity of notation, we will often omit in this proof the dependence on $\tau$ of all our objects, until we need to let $\tau\to 0$. For all $k$, one also defines a vector field $v_k$, given by $v_k=({\rm id}-T)/\tau=\nabla \varphi_k/\tau$, where $T$ is the optimal transport map from $f_k$ to $f_{k-1}$ and $\varphi_k$ is the corresponding Kantorovich potential. As previously notes, the optimality conditions on $f_k$ allow us to check that we have $v_k=-\nabla (U'(f_k/m))$.

Defining $E_k=f_kv_k,$ one has $E_k= c(r)m \nabla (f_k/m)^{-r}$ for some $c(r)>0$. With $f_k$ and $E_k$ one can define a piecewise constant interpolation $(f^{(\tau)},E^{(\tau)})$ (which of course depends on the value of the parameter $\tau$), satisfying $E^{(\tau)}=-(r+1)m \nabla (f^{(\tau)}/m)^{-r}$. It is also possible to define a piecewise geodesic interpolation $(\hat f^{(\tau)},\hat E^{(\tau)})$, where $t\mapsto\hat f^{(\tau)}(t)$ is continuous for the $W_2$ distance, $\hat f^{(\tau)}(k\tau)=f^{(\tau)}_k$, and $\hat f^{(\tau)}(t)$ is given for all $t\in [k\tau,(k+1)\tau]$ by a geodesic in the $2$-Wasserstein space, parameterised by constant speed, and $\hat E^{(\tau)}=\hat f^{(\tau)}\hat v^{(\tau)},$ where $\hat v^{(\tau)}$ is the corresponding optimal velocity field, so that we have $\partial_t \hat f^{(\tau)}+\div_x \hat E^{(\tau)}=0$. 

It is standard from the theory of gradient flows in Wasserstein space (see \cite{AGS,arabsurvey} and \cite[Chapter 8]{OTAM}) to prove that $f^{(\tau)},E^{(\tau)},\hat f^{(\tau)}$ and $\hat E^{(\tau)}$ admit weak limits, up to subsequences, when $\tau\to 0$, and we call these limits $f,E,\hat f$ and $\hat E$, respectively. It is also standard that we have $f=\hat f$ and $E=\hat E$, and $\partial_t \hat f+\div_x\hat E=0$. One is only left with proving that we have $E=-(r+1)m \nabla (f/m)^{-r}$. 

Since this is a nonlinear relation, it cannot be directly deduced by the weak convergence. However, we have $f^{(\tau)}(t)\destar f(t)$, hence $f^{(\tau)}(t)/m\destar f(t)/m$ for each $t\geq0$, and we have a uniform $BV$ bound provided by Lemma \ref{lem:BV}. This transforms the weak convergence into a strong one, and we thus have almost-everywhere convergence. Together with the uniform bounds from above and below (a consequence of Lemma \ref{lem:max}), this provides the convergence of $[f^{(\tau)}(t)/m]^{-r}$ to $ [f(t)/m]^{-r}$ as $\tau\to0$ and takes care of the nonlinearity. Then we apply Lemma \ref{H1estimate} to $g=f_k^{(\tau)}$ and $f_*=f_{k+1}^{(\tau)},$ which provides, summing up over $k$ for $q>1$,
\begin{equation}\label{H1estimate-semicont}
c(r,q)\int_0^T \int_\Omega m \left|\nabla\left(\frac{f^{(\tau)}}{ m}\right)^{\frac{q-(r+1)}{2}}\right|^2\leq \G_{(q)}[f_0]+(q-1)\Lambda \left(\frac{C_0\sup m}{\inf m}\right)^{q-1}\sum_k W_2^2(f^{(\tau)}_k,f^{(\tau)}_{k+1}).
\end{equation}
Also, applying Lemma \ref{H1estimate-JKO} to $g=f_k^{(\tau)}$ and $f_*=f_{k+1}^{(\tau)},$ and summing over $k,$ we get 
$$
\sum_k W_2^2(f^{(\tau)}_k,f_{k+1})\leq 2\tau \sum_k \(\mathcal F_\rho(f^{(\tau)}_k)-\mathcal F_\rho(f_{k+1})\)\leq 2\tau \mathcal F_\rho[f_0].
$$
This yields uniform $H^1$ bounds. In particular, we choose $q=r+3$ in \eqref{H1estimate-semicont} and we obtain a uniform $L^2$ bound, in time and space, on $\nabla (f^{(\tau)}/m)$. Using the lower and upper bounds on the ratio $f^{(\tau)}/m$ this also translates into a similar bound on $[f^{(\tau)}(t)/m]^{-r}$ and allows to pass to the limit. 

Therefore, under the assumptions of this lemma, we have the existence of a weak solution with intial datum $f_0$. Moreover this solution satisfies 
$$
\frac{f}{m} \in L^2([0,\infty),\dot H^1(\Omega)),\qquad 
\Bigl(\frac{f}{m}\Bigr)^{-r} \in L^2([0,\infty),\dot H^1(\Omega)),
$$
where the first $L^2$ norm is only bounded in terms of $\G_{(r+3)}[f_0]$ (the dependence on the constants $c_0,C_0$ and $\Lambda$ disappears in the limit $\tau\to 0$). On the other hand, the second bound depends on  $c_0$ (since the Lipschitz constant of the function $s\mapsto s^{-r}$ depends on the lower bounds; yet, it is possible to obtain uniform bounds using Lemma \ref{H1estimate-JKO}).
\end{proof}

We now relax the extra assumptions of Lemma \ref{thm:existence-weak-1} and get the existence part of Theorem \ref{thm:existence-uniqueness}. In fact, let us restate it in a slightly more precise way:

\begin{thm}[Existence of weak solutions]\label{thm:existence-weak}
Suppose that $f_0\in L^{r+3}(\Omega)$ with $\mathcal F_\rho[f_0]<\infty$. Then there exists a distributional solution of \eqref{eq:f} starting from $f_0$ which satisfies
$$
\frac{f}{m} \in L^2([0,\infty),\dot H^1(\Omega))\cap L^2_{\rm loc}([0,\infty),H^1(\Omega)),\quad 
\Bigl(\frac{f}{m}\Bigr)^{-r} \in L^2([0,\infty),\dot H^1(\Omega))\cap L^2_{\rm loc}([0,\infty),H^1(\Omega)),
$$
and is therefore a weak solution according to Definition \ref{def:weak}. Morever, if $f_0\geq c_0 m$ for some $c_0>0$, then this solution also satisfies $f(t)\geq c_0m $ for every $t\geq 0$.
\end{thm}

 \begin{proof}
 We proceed by approximation, considering a sequence of initial data $(f_{n,0})_n$ and a sequence of weights $(m_n)_n$. If we suppose that the sequences $(f_{n,0})_n$ and $(m_n)_n$ satisfy the assumptions of Lemma \ref{thm:existence-weak-1}, then we have a sequence of solutions $(f_n)_n$. The approximation is chosen so that for all $n$ the Sobolev norm and the upper and lower bounds of $m_n$ are preserved, as well as the bounds on $\G_{(r+3)}[f_{n,0}]$ and $\mathcal F_\rho[f_{n,0}]$. We then apply Lemma \ref{lem:stability}, since the functions $f_n$ satisfy all the required assumptions.
  
  The last part of the statement (i.e. preservation of the lower bounds of $f/m$) is a direct consequence of the approximation, provided $f_{n,0}$ is chosen so that $f_{n,0}\geq c_0m$ for all $n$.
%
%
%
%
 \end{proof}

\subsection{Weighted $L^1$ contractivity and uniqueness}
We now prove the uniqueness part Theorem \ref{thm:existence-uniqueness} by showing a  contractivity result on weak solutions.
As the reader will see from the proof, this actually follows by a weighted contractivity result on the equation satisfied by $f/m$.
Although weighted contractivity estimates have already appeared in the literature (see for instance \cite{Va,BV15,BV16}),
these are completely new in the ultrafast regime setting, and we expect both the result and the method of the proof to be useful in other circumstances.

 \begin{prop}[$L^1$ contractivity]
 \label{prop:contr}
	Let $f$ and $g$ be two nonnegative weak solutions of \eqref{eq:f}, and suppose that there exists a constant $c_0>0$ such that $g(t)\ge c_0m>0$ for all $t\in [0,\infty)$.
Then,
$$
	\int_{\Omega}(f(t,x)-g(t,x))_+\di x\le \int_{\Omega}(f(0,x)-g(0,x))_+\di x\qquad \text{for all}\ \,t \geq 0,
$$
$$
	\int_{\Omega}(f(t,x)-g(t,x))_-\di x\le \int_{\Omega}(f(0,x)-g(0,x))_-\di x\qquad \text{for all}\ \,t \geq 0,
$$
which in particular implies $L^1$ contractivity:
$$
	\int_{\Omega} |f(t,x)-g(t,x)| \di x \le \int_{\Omega}|f(0,x)-g(0,x)| \di x\qquad \text{for all}\ \,t \geq 0.
$$
 \end{prop}
 
 \begin{proof}
As often in this paper, it is convenient to use the notation $u=f/m$ and $v=g/m$.
  
For every $\eps>0$, let us consider $\psi_\eps \colon \R \to \R$ to be a smooth approximation of the positive part defined as follows: for all $s \in \R$, 
 $$
\psi_\eps(s)=\left\{
\begin{array}{ccc}
 0 & \mbox{if}\ s<-\eps,\\
 \frac{1}{4\eps}(s+\eps)^2 & \mbox{if}\  -\eps\le s\le\eps, \\
 s & \mbox{if}\ s>\eps.
 \end{array}
 \right.
 $$
We fix $t \geq 0$. Using Lemma \ref{lem:W1infty} we deduce that $\pt_t(u-v) \in L^2_{\rm loc}([0,\infty),H^{-1}(\Omega))$, which allows us to justify the next computation:
\be\label{eq:psi1}
\begin{split}
	&\frac{\di}{\di t}\int_{\Omega}m(x)\,\psi_\eps(u(t,x)-v(t,x))\di x=\int_{\Omega}m(x)\,\psi'_\eps(u(t,x)-v(t,x))\pt_t(u-v)(t,x)\di x\\
&\phantom{=}\overset{\eqref{eq:u}}{=}-(r+1)\int_{\Omega}\[ \psi'_\eps(u(t,x)-v(t,x)) \]\\
&\phantom{ =\overset{\eqref{eq:u}}{=}-(r+1)\int_{\Omega} [ \psi'_\eps(u(t,x) }
	 \times\[\div_x(m(x)\nabla_x (u(t,x))^{-r})-\div_x(m(x)\nabla_x (v(t,x))^{-r})\]\di x.
\end{split}
\ee
 Integrating by parts,
 \begin{align*}
&\frac{\di}{\di t}\int_{\Omega}m(x)\,\psi_\eps(u(t,x)-v(t,x))\di x\\
&\phantom{=}=(r+1)\int_{\Omega}\,m(x)\,\psi''_\eps(u(t,x)-v(t,x))\,\nabla_x(u(t,x)-v(t,x))\cdot\[\nabla_x u(t,x)^{-r}-\nabla_x v(t,x)^{-r}\]\di x\\
&\phantom{=}=-r(r+1)\int_{\Omega} \[m(x)\,\psi''_\eps(u(t,x)-v(t,x))\,\nabla_x(u(t,x)-v(t,x))\]\\
&\phantom{ == -r(r+1)\int_{\Omega} [ m(x)\,\psi''_\eps(u(t,x)-v(t,x))\, } 
	\cdot\[u(t,x)^{-(r+1)}\nabla_x u(t,x)-v(t,x)^{-(r+1)}\nabla_x v(t,x)\]\di x.
\end{align*}
Adding and subtracting $u(t,x)^{-(r+1)}\nabla_x v(t,x)$ in the square brackets above we get
\begin{align*}
	&\frac{\di}{\di t}\int_{\Omega}m(x)\,\psi_\eps(u(t,x)-v(t,x))\di x\\
&\phantom{=}=-r(r+1)\int_{\Omega}\,m(x)\,\psi''_\eps(u(t,x)-v(t,x))\,u(t,x)^{-(r+1)}\,|\nabla_x(u(t,x)-v(t,x))|^2 \di x\\
&\phantom{==} -r(r+1)\int_{\Omega} \left[ m(x)\,\psi''_\eps(u(t,x)-v(t,x))\(u(t,x)^{-(r+1)}-v(t,x)^{-(r+1)}\) \right]\\
&\phantom{ === -r(r+1)\int_{\Omega} [ m(x)\,\psi''_\eps(u(t,x)-v(t,x)) }
	\times \left[ \nabla_x\(u-v\)(t,x) \cdot\nabla_x v(t,x)\right] \di x\\
&=: I_1+I_2.
\end{align*}
The first term $I_1$ is nonpositive because $\psi_\eps'' \geq0$.
By definition of $\psi_\eps$ we have
\begin{align*}
	&I_2 = -r(r+1)\int_{\{|u(t,x)-v(t,x)|\le \eps\}} \left[ \frac{m(x)}{2\eps}\(u(t,x)^{-(r+1)}-v(t,x)^{-(r+1)}\) \right]\\
	&\phantom{== -r(r+1)\int_{\{|u(t,x)-v(t,x)|\le \eps\}} [ \frac{m(x)}{2\eps}}
		\times \left[\nabla_x\(u-v\)(t,x)\cdot\nabla_x v(t,x) \right] \di x.
\end{align*}
Note that, for some $C_r>0$,
 $$
 |u^{-(r+1)}-v^{-(r+1)}|=\frac{|v^{r+1}-u^{r+1}|}{u^{r+1}v^{r+1}}\leq C_r\frac{|u-v|(u^r+v^r)}{u^{r+1}v^{r+1}}.
 $$
Thus, for a constant $C>0$ (which in the rest of the proof may change value across any line),
\begin{align*}
	I_2 &\le C\int_{\{|u(t,x)-v(t,x)|\le \eps\}}\frac{m(x)}{\eps}\frac{|u(t,x)-v(t,x)|(u^r(t,x)+v^r(t,x))}{u^{r+1}v^{r+1}(t,x)}\big|\nabla_x\(u-v\)(t,x)\big|\cdot|\nabla_x v(t,x)|\di x\\
	&\leq C\int_{\{|u(t,x)-v(t,x)|\le \eps\}}m(x)\frac{u^r(t,x)+v^r(t,x)}{u^{r+1}v^{r+1}(t,x)}\big|\nabla_x\(u-v\)(t,x)\big|\cdot|\nabla_x v(t,x)|\di x.
\end{align*}
From \eqref{eq:psi1} we therefore get
\be\label{eq:psi2}
\begin{split}
	&\frac{\di}{\di t}\int_{\Omega}m(x)\,\psi_\eps(u(t,x)-v(t,x))\di x\\
	&\phantom{=}\leq C\int_{\{|u(t,x)-v(t,x)|\le \eps\}}m(x)\frac{u^r(t,x)+v^r(t,x)}{u^{r+1}v^{r+1}(t,x)}\big|\nabla_x\(u-v\)(t,x)\big|\cdot|\nabla_x v(t,x)|\di x.
\end{split}
\ee
We now claim that the integrand in the right-hand side above belongs to $L^1([0,\infty)\times\Omega)$. To prove this, recall that, by assumption, we have $v \ge c_0$. Since in the domain of integration in $I_2$ we have $|u(t)-v(t)|\le\eps,$ for $\eps$ small enough (for instance $\eps \leq c_0/4$) we also have that $u(t)\ge c_0/2$. Therefore, on the domain of integration in $I_2$,
\begin{equation} \label{eq:u v}
	\frac{u(t)}{2}\leq v(t)\leq 2u(t).
\end{equation}
Also note that, on the domain of integration (dropping the $(t,x)$ dependences to simplify the following computation),
\begin{align*}
m \frac{(u^r+v^r)}{u^{r+1}v^{r+1}}\big|\nabla_x\(u-v\)\big|\cdot|\nabla_x v| &\leq m\frac{(u^r+v^r)}{u^{r+1}v^{r+1}}\big(|\nabla_x  u|+|\nabla_x v|\big)\cdot|\nabla_x v|\\
&\leq 2 \,m \frac{(u^r+v^r)}{u^{r+1}v^{r+1}}\Big(|\nabla_x u|^2+|\nabla_x v|^2\Big)\\
&\leq C \Bigl(\frac{u^r}{u^{2(r+1)}}|\nabla_x u |^2+\frac{v^r}{v^{2(r+1)}}|\nabla_x v|^2\Big) m\\
&= C \Bigl(\frac{1}{u^{r+2}}|\nabla_x u|^2+\frac{1}{v^{r+2}}|\nabla_x v|^2\Big) m\\
&= C\Bigl(|\nabla_x u^{-r/2}|^2+|\nabla_x v^{-r/2}|^2\Big) m,
\end{align*}
where the last inequality follows from \eqref{eq:u v}. By the definition of weak solutions, we have that $u,v,u^{-r},v^{-r} \in L^2([0,\infty),\dot H^1(\Omega))$, so that $u^{-r/2},v^{-r/2} \in L^2([0,\infty),\dot H^1(\Omega))$ by Lemma \ref{lem:holder}. This proves that 
\begin{equation}
\label{eq:integrable}
	m \frac{(u^r+v^r)}{u^{r+1}v^{r+1}}\big|\nabla_x\(u-v\)\big|\cdot|\nabla_x v| \in L^1([0,\infty)\times \Omega).
\end{equation}
We now integrate in time the differential inequality \eqref{eq:psi2}: given $T>0,$ we have
\begin{align*}
\int_{\Omega}&m(x)\psi_\eps(u(T,x)-v(T,x))\di x\le \int_{\Omega}m(x)\psi_\eps(u(0,x)-v(0,x))\di x\\
&+C\int_0^T\int_{\{|u(t,x)-v(t,x)|\le \eps\}}m(x)\frac{(u^r(t,x)+v^r(t,x))}{u^{r+1}v^{r+1}(t,x)}\big|\nabla_x\(u-v\)(t,x)\big\|\nabla_x v(t,x)|\di x\di t,
\end{align*}
and, by dominated convergence (thanks to \eqref{eq:integrable}), in the limit $\eps \to 0$ we obtain the following:
\begin{align*}
\int_{\Omega}&m(x)(u(T,x)-v(T,x))_+\,\di x\le \int_{\Omega}m(x)(u(0,x)-v(0,x))_+\di x\\
&+C\int_0^T\int_{\{u(t,x)=v(t,x)\}}m(x)\frac{(u^r+v^r)}{u^{r+1}v^{r+1}}\big|\nabla_x\(u-v\)(t,x)\big| |\nabla_x v(t,x)|\di x\di t.
\end{align*}
Since on the region $\{u(t,x)=v(t,x)\}$ we have that $\nabla_xu(t,x)=\nabla_xv(t,x)$ for almost every $x$
(see for instance \cite[Theorem 4, Chapter 4.2.2]{EvaGar}), we get that 
$$
	\int_{\Omega}m(x)(u(t,x)-v(t,x))_+\di x\le \int_{\Omega}m(x)(u(0,x)-v(0,x))_+\di x\qquad \text{for all}\ \,t>0.
$$
Recalling that $u=f/m$ and $v=g/m$, this proves the first part of the statement.

Repeating the proof, this time with $\psi_\eps$ a smooth approximation of the negative part,
we obtain
$$
\int_{\Omega}m(x)(u(t,x)-v(t,x))_-\,\di x\le \int_{\Omega}m(x)(u(0,x)-v(0,x))_-\,\di x\qquad \text{for all}\ \,t>0,
$$
concluding the proof.
\end{proof}

Let us now restate the uniqueness part of Theorem \ref{thm:existence-uniqueness} (which implicitly requires the usual standing assumption on $m$):

 \begin{thm}[Uniqueness of weak solutions] \label{thm:uniq sol}
	Given $f_0 \in L^{r+3}(\Omega)$ with $\f_\rho[f_0]<\infty$, there exists at most one weak solution of \eqref{eq:f} starting from $f_0$.
 \end{thm}
\begin{proof}
Suppose that $f$ is a weak solution starting from $f_0$, and write $u=f/m$.
Set $u_0=f_0/m$, consider $u_0^\eps=\max\{u_0,\eps\}$, and let $v_\eps$ be the solution constructed in Theorem \ref{thm:existence-weak} starting from $u_0^\eps$ for all $\eps>0$. Note that we have $v_\eps \geq \eps>0$. By Proposition \ref{prop:contr} we deduce that
$$
\int_\Omega |u(t,x)-v_\eps(t,x)|\,\di x \leq \int_\Omega |u_0(x) -v_\eps(0,x)|\di x = \int_\Omega |u_0(x) -u_0^\eps(x)|\di x \leq C\eps \qquad \text{for all}\ \,t \geq 0,
$$
for some time-independent constant $C>0$. Letting $\eps \to0$, we get
$$
\int_\Omega |u(t,x)-v(t,x)|\di x \leq 0\qquad \text{for all}\ \,t \geq 0,
$$
where $v$ is an arbitrary limit of $(v_\eps)_\eps$. Hence $u(t)$ must coincide with $v(t)$ and therefore, since $u$ (resp. $f$) can be any weak solution starting from $u_0$ (resp. $f_0$), we obtain uniqueness.
\end{proof}

We conclude this section by showing some easy and useful corollaries. First, we give a continuous maximum principle which is a corollary of Proposition \ref{prop:contr}. Remark \ref{rem:corollary} below also shows that, in fact, this continuous principle can be seen as a corollary of the discrete maximum principle (Lemma \ref{lem:max}) and Theorem \ref{thm:uniq sol}.

\begin{cor}[Continuous maximum principle]
\label{cor:max princ}
Let $f$ be a weak solution of \eqref{eq:f} starting from some initial datum $f_0$ such that
$f_0\leq C_0m $ (resp. $f_0\geq c_0m$).
Then $f(t)\leq C_0m $ (resp. $f(t)\geq c_0m$) for all $t\geq 0$.
\end{cor}
\begin{proof}
	Assume for instance that $f_0\leq C_0m $ (the case $f_0\geq c_0m$ being analogous).
Then we apply Proposition \ref{prop:contr} to $f(t)$ and $g(t)=C_0m$ to deduce that
$$
\int_\Omega (f(t,x)-C_0m(x))_+ \di x\leq 
\int_\Omega (f(0,x)-C_0m(x))_+\di x=0\qquad \text{for all}\ \,t\geq 0.
$$ 
Thus, for all $t\geq0$ we get $f(t)\leq C_0m$, as desired.
\end{proof}

\begin{remark}\label{rem:corollary}
	The above corollary can be also proved by noticing that, thanks to the discrete maximum principle, it holds for all solutions obtained as limit of the JKO scheme. Since by uniqueness all solutions can be obtained in this way, the result follows.
\end{remark}

Finally, we give a useful remark, which is now straightforward, about the continuous dependence of the unique weak solution in terms of the initial data.

\begin{cor}[Stability with respect to the initial data]
\label{cor:stability}
Suppose that $(f_{n,0})_n$ is a sequence of initial data bounded in $L^{r+3}(\Omega)$ with $\f_\rho[f_{0,n}]<\infty$ for all $n$, and such that the sequence $(f_{n,0}^{-1})_n$ is bounded in $L^r(\Omega)$. Suppose that for each $n$ the function $f_{n,0}$ is bounded (not necessarily uniformly in $n$) from above and from below by positive constants. Suppose $f_{n,0}\destar f_0$ as $n\to\infty$. Then, the unique weak solution $f_n$ associated with the initial datum $f_{n,0}$ converges (weakly for every time, and weakly in $L^2([0,T],H^1(\Omega))$ for every $T$) as $n\to\infty$ to the unique weak solution associated with the initial datum $f_{0}$.
\end{cor}
\begin{proof}
	This is a consequence of Lemma \ref{lem:stability} and of the uniqueness result of Theorem \ref{thm:uniq sol}. Note that there is no need to extract a subsequence because of the uniqueness.
\end{proof}

\section{Harnack inequalities: proof of Theorem \ref{thm:bddness-longtime}}\label{sec:reg}

In \cite{Iac2} the author proves exponential convergence to equilibrium for initial data that are bounded away from zero and infinity (although the result there is stated only for $d=1$, the proof works without modification in any dimension). Exponential convergence results will be the object of Section \ref{sec:long-time-behaviour} and will be based on the preliminary proof of the fact that, instantaneously, solutions become bounded from above and below.  The following proposition, along with the continuous maximum principle stated in Corollary \ref{cor:max princ}, gives the proof of Theorem \ref{thm:bddness-longtime}. Note that the hypotheses on $f_0$ as stated in Theorem \ref{thm:bddness-longtime} are equivalent to those as stated in Proposition \ref{prop:reg1} below.

 \begin{prop}[Instantaneous regularisation; Harnack inequalities]\label{prop:reg1}
Let $T\in (0,1]$ and assume $\log m\in W^{1,p}(\Omega)$ for some $p>d$. Let $f_0 \in L^{r+3}(\Omega)$ with $\mathcal F_{\rho}[f_0]<\infty$, and let $f$ be a weak solution of \eqref{eq:f} starting from $f_0$. Let us write $\sigma=r+1$. Assume that there exists $q> \sigma\max\(1,\frac{d}{2}\)$ such that $f_0,f_0^{-1} \in L^q(\Omega)$. Then there are constants $C_1,C_{-1}>0$, independent of $T$, such that
\be\label{eq:estimate-nu}
 	\| f(3T)\|_{L^\infty(\Omega)}\le \frac{C_1}{T^{\alpha}}\|f_0\|^\beta_{L^q(\Omega)},
\ee
\be\label{eq:estimate-nu2}
 	\| f(3T)^{-1}\|_{L^\infty(\Omega)}\le \frac{C_{-1}}{T^{\alpha(1+2\alpha\sigma)}}\|f_0^{-1}\|^\beta_{L^q(\Omega)}\|f_0\|^{2\alpha\sigma\beta}_{L^q(\Omega)},
\ee
 where $\alpha=A_\infty B_\infty$ and $\beta=B_\infty$
 are given by
 $$
 A_\infty:=\sum_{i=1}^\infty \frac{1}{\bar q_i},\qquad B_\infty:=\prod_{i=0}^{\infty}\frac{q_{i}}{\bar q_{i}}.
 $$
Here, for all $i \in \mathbb{N}\cup\{0\}$, $q_i=\theta^i q-\theta\sigma\frac{\theta^i-1}{\theta-1}$, $\theta=\frac{d}{d-2}$, and $\bar q_i=q_i-\sigma$.
 \end{prop}

 \begin{proof}
	It is convenient to prove the regularisation result in terms of $u:=f/m$ and then, at the end, rewrite the result in terms of $f$. We also write $u_0:=u(0)=f_0/m$. We proceed as follows: we first obtain two inequalities, one for the gradient of $u$ and the other for $u$ (see \eqref{eq:H1} and \eqref{eq:L2}); we then show the result \eqref{eq:estimate-nu} for $f$ using a Moser iteration; we finally prove \eqref{eq:estimate-nu2} for $f^{-1}$ in a similar way, using in fact the result for $f$. All our estimates have to be considered as a priori estimates: in order to perform the computations, in particular the derivations of certain integrals in time, we need Sobolev bounds on powers of $u$; therefore, we can start by supposing that our initial datum is bounded from above and below, which implies, thanks to Corollary \ref{cor:max princ}, that the same bounds propagate to every time $t$ and the regularity $u\in L^2_{\rm loc}([0,\infty),H^1(\Omega))$ implies the same regularity for all powers of $u$. Then, we note that the estimates we obtain do not depend on the bounds on the initial datum, and we therefore use Corollary \ref{cor:stability} to deduce by approximation the same estimates for general initial data. We will not make this procedure explicit in this proof.
\medskip

{\it Step 1: inequalities for $u$ and its gradient.} Let $\nu \in \{-1,1\}$. We differentiate the weighted $L^q$ norm of $u^\nu$ (in a similar computation as that to obtain \eqref{H1estimate-semicont2}), and then use the equation for $u$ and an integration by parts to obtain, for all $t \in [0,T]$,
 \begin{equation}\label{eq:der Lq}
 \begin{split}
 	\frac{\di}{\di t}&\int_{\Omega} u(t,x)^{\nu q} m(x)\di x=\nu q\int_{\Omega}u(t,x)^{\nu q-1}\pt_tu(t,x)m(x)\di x\\
 &=-\nu q\sigma\int_{\Omega}u(t,x)^{\nu q-1}\div_x\(m(x)\nabla_x \(u(t,x)^{-r}\)\)\di x\\
 &=\nu q\sigma\int_{\Omega} m(x)\,\nabla_x\(u(t,x)^{\nu q-1}\) \cdot \nabla_x\( u(t,x)^{-r}\)\di x\\
 &=-\nu q\,(\nu q-1)\,r\sigma\int_{\Omega}m(x)\left|\nabla_x\left( u(t,x)^\nu \right)\right|^2u(t,x)^{\nu (q-2)-\sigma}\di x \leq 0.
\end{split}
\end{equation}
If we define $C_\nu(q)>0$ such that
$$
	C_\nu(q)^2=\frac{4\,\nu q\,(\nu q-1)\,r\,\sigma}{(q-\sigma)^2},
$$ 
and we notice that
$$
	\left|\nabla_x\left(u(t,x)^\nu\right)\right|^2u(t,x)^{\nu(q-2)-\sigma}=\frac{1}{\eta^2}\left|\nabla_x\bigl(u(t,x)^{\nu\eta}\bigr) \right|^2 u(t,x)^{(\nu-1)\sigma},\qquad \eta:=\frac{q-\sigma}{2},
$$
where $\eta>0$ thanks to the condition $q>\sigma$, we obtain
\bes
	\frac{\di}{\di t} \int_{\Omega} m(x)u(t,x)^{\nu q}\di x=-C_\nu(q)^2 \int_{\Omega}m(x)\left|\nabla_x\bigl(u(t,x)^{\nu \eta}\bigr)\right|^2 u(t,x)^{(\nu-1)\sigma}\di x.
\ees
Set $t_0:= T$. Integrating the previous expression between $t_0$ and $t_0+T$ and dividing by $T$ we have
\begin{equation*}
	\frac{C_\nu(q)^2}{t_0} \int_{t_0}^{t_0+T}\int_{\Omega}m(x)\left|\nabla_x\bigl(u(t,x)^{\nu \eta}\bigr)\right|^2 u(t,x)^{(\nu-1)\sigma} \di x\di t\le\frac{1}{T}\int_{\Omega}\,m(x)u(t_0,x)^{\nu q}\di x.
\end{equation*}
Thus, there exists $\bar t\in\(t_0,t_0+T\)$ such that
\be\label{eq:H1}
\begin{split}
	\int_{\Omega}\left|\nabla_x\bigl(u(\bar t,x)^{\nu \eta}\bigr) \right|^2 u(\bar t,x)^{(\nu-1)\sigma} \di x &\le\frac{1}{\lambda\,C_\nu(q)^2\,T}\int_{\Omega} m(x) u(t_0,x)^{\nu q}\di x\\
	&\le\frac{1}{\lambda^2\,C_\nu(q)^2\,T}\int_{\Omega}u(t_0,x)^{\nu q}\di x ,
\end{split}
\ee
using $\lambda \le m\le 1/\lambda$; see the Sobolev standing assumption on $m$. Note that since the weighted $L^q$ norm of $u^\nu$ is nonincreasing in time (see \eqref{eq:der Lq}), we know $u(t_0)^\nu \in L^q(\Omega)$. Furthermore, following the same computation as in \eqref{eq:der Lq} replacing $q$ by $2\eta$, the weighted $L^{2\eta}$ norm of $u^\nu$ is as well nonincreasing, and therefore, because again $\lambda \le m\le 1/\lambda$,
$$
\int_{\Omega}u(\bar t,x)^{2\nu \eta}\di x
\leq \frac{1}{\lambda}\int_{\Omega}u(\bar t,x)^{2\nu\eta}m(x)\di x
\leq \frac{1}{\lambda}\int_{\Omega}u(t_0,x)^{2\nu\eta}m(x)\di x
\leq \frac{1}{\lambda^2}\int_{\Omega}u(t_0,x)^{2\nu\eta}\di x.
$$
Since $2\eta< q$, by H\"older's inequality we obtain
\bes
	\int_{\Omega}u(t_0,x)^{2\nu\eta}\di x \leq |\Omega|^{1-\frac{2\eta}{q}} \(\int_{\Omega}u(t_0,x)^{\nu q}\di x\)^{2\eta/q},
\ees
which finally gives
\be\label{eq:L2}
	\int_{\Omega}u(\bar t,x)^{2\nu \eta}\di x \leq \frac{|\Omega|^{\frac\sigma q}}{\lambda^2} \(\int_{\Omega}u(t_0,x)^{\nu q}\di x\)^{2\eta/q}.
\ee
\medskip

{\it Step 2: proof of \eqref{eq:estimate-nu} ($\nu =1$).} Since $u(t_0) \in L^q(\Omega),$ \eqref{eq:H1} and \eqref{eq:L2} imply that $u(\bar t)^{\eta}\in H^1(\Omega)$.
Thus, by Sobolev's inequality, when $d\geq 3$ we obtain
\begin{align*}
	\(\int_{\Omega}u(\bar t,x)^{2^*\eta}\di x\)^{1/2^*}\le C_\mathrm{S}\(\(\int_{\Omega}u(\bar t,x)^{2\eta}\di x\)^{1/2}+\(\int_{\Omega}\left|\nabla_x\bigl(u(\bar t,x)^{\eta}\bigr)\right|^2\di x\)^{1/2} \),
\end{align*}
where $C_\mathrm{S}>0$ is the Sobolev constant (depending on $\Omega$) and $2^*:=2d/(d-2)$. When $d\in \{1,2\}$ the same inequality holds by replacing $2^*$ with any number larger than two\footnote{For the one-dimensional case it is actually even easier, since we can see that the $H^1$ estimate obtained so far is itself enough to provide $L^\infty$ bounds.}. Because this does not change any argument given in the rest of the proof, for simplicity and without loss of generality we assume $d\geq 3$.
Using \eqref{eq:H1} and \eqref{eq:L2} with $\nu=1$ and $\lambda \le m\le 1/\lambda,$ 
$$
	\(\int_{\Omega}u(\bar t,x)^{2^*\eta}\di x\)^{1/2^*}\le \frac{C_\mathrm{S}}{\lambda} \( |\Omega|^{\frac{\sigma}{2q}}\(\int_{\Omega}u(t_0,x)^{q}\di x\)^{\eta/q}+\frac{1}{C_1(q)\,\sqrt{T}}\(\int_{\Omega}\,u(t_0,x)^{q}\di x\)^{1/2} \).
$$
Recalling that $m\le 1/\lambda,$ we get
\be\label{eq:M}
	\|u(t_0)\|_{L^q(\Omega)}^q\geq |\Omega|^{1-q} \|u(t_0)\|_{L^1(\Omega)}^q=|\Omega|^{1-q}\(\int_{\Omega} \frac{f(t_0,x)}{m(x)} \di x\)^q \geq |\Omega|^{1-q}(\lambda  M)^q,
\ee
and so $\frac{|\Omega|^{q-1}}{(\lambda M)^q}\int_{\Omega}u(t_0,x)^{q}\di x\geq 1$. Therefore, since $\eta/q <1/2$,
\begin{align*}
	\(\int_{\Omega}u(t_0,x)^{q}\di x\)^{\eta/q}&= |\Omega|^{\frac{(1-q)\eta}{q}} (\lambda M)^\eta\(\frac{|\Omega|^{q-1}}{(\lambda M)^q}\int_{\Omega}u(t_0,x)^{q}\di x\)^{\eta/q}\\
	&\leq |\Omega|^{\frac{(1-q)\eta}{q}} (\lambda M)^\eta\(\frac{|\Omega|^{q-1}}{(\lambda M)^q}\int_{\Omega}u(t_0,x)^{q}\di x\)^{1/2}\\
	&=|\Omega|^{(1-q)\left( \frac\eta q - \frac12\right)} (\lambda M)^{\eta-\frac q2} \(\int_{\Omega}u(t_0,x)^{q}\di x\)^{1/2}.
\end{align*}
Hence, all in all,
\bes
	\(\int_{\Omega}u(\bar t,x)^{2^*\eta}\di x\)^{1/2^*}\le \frac{C_\mathrm{S}}{\lambda} \(\frac{|\Omega|^{\frac\sigma2}}{(\lambda M)^{\frac\sigma2}}\(\int_{\Omega}u(t_0,x)^{q}\di x\)^{1/2}+\frac{1}{C_1(q) \sqrt{T}}\(\int_{\Omega}\,u(t_0,x)^{q}\di x\)^{1/2} \).
\ees
Using that $T \leq 1$ then gives
\be\label{eq:estimate-star}
	\(\int_{\Omega}u(\bar t,x)^{2^*\eta}\di x\)^{1/2^*}\leq \frac{\tilde C_1(q)}{\sqrt T}\(\int_{\Omega}\,u(t_0,x)^{q}\di x\)^{1/2},
\ee
where
\bes
	\tilde C_1(q) = \frac{C_\mathrm{S}}{\lambda} \( \frac{|\Omega|^{\frac\sigma2}}{(\lambda M)^{\frac\sigma2}} + \frac{1}{C_1(q)} \) >0.
\ees
Note that $\tilde C_1(p) \to \frac{C_\mathrm{S}}{\lambda} \left( \frac{|\Omega|^{\frac\sigma2}}{(\lambda M)^{\frac\sigma2}} + \frac{1}{2\sqrt{r\sigma}} \right)$ as $p\to\infty$.

We now want to initialise a Moser interative scheme. To this end, let us define $\eta_0=\eta$, $q_0=q$, $q_1=2^*\eta_0$, and $t_1=\bar t$.
With this notation, \eqref{eq:estimate-star} yields
$$
	\(\int_{\Omega}u(t_1,x)^{q_1}\di x\)^{\eta_0/q_1}\leq  \frac{\tilde C_1(q_0)}{\sqrt{T}}\(\int_{\Omega}\,u(t_0,x)^{q_0}\di x\)^{1/2},
$$
or equivalently
$$
\|u(t_1)\|^{\eta_0}_{L^{q_1}(\Omega)}\le \frac{\tilde C_1(q_0)}{\sqrt T}\| u(t_{0})\|^{\frac{q_0}{2}}_{L^{q_0}(\Omega)}.
$$
Observe that $q_1>q_0$ thanks to the assumption that $q_0>\frac{\sigma d}{2}$. We can now repeat the argument above starting from $t_1$ in place of $t_0$, $q_1$ in place of $q_0$, $T/2$ in place of $T$, and we find a time $t_2 \in (t_1,t_1+T/2)$ such that
 $$
	\|u(t_2)\|^{\eta_1}_{L^{q_2}(\Omega)}\le \frac{2^{1/2} \tilde C_1(q_1)}{\sqrt{T}}\| u(t_{1})\|^{\frac{q_{1}}{2}}_{L^{q_{1}}(\Omega)}, \qquad \eta_1 = \frac{q_1-\sigma}{2}, \quad q_2=2^*\eta_1.
$$
Iterating $k \in \mathbb{N}$ times, we find
$$
t_k\in \(t_{k-1}, t_{k-1}+\frac{T}{2^{k-1}}\)
$$
such that
\begin{equation}\label{eq:stepk}
	\|u(t_k)\|^{\eta_{k-1}}_{L^{q_k}(\Omega)}\le \frac{2^{(k-1)/2}\tilde C_1(q_{k-1})}{\sqrt{T}}\| u(t_{k-1})\|^{\frac{q_{k-1}}{2}}_{L^{q_{k-1}}(\Omega)}, \qquad \eta_{k-1} = \frac{q_{k-1}-\sigma}{2},
\end{equation}
where
$$
	q_k=2^*\eta_{k-1}=\frac{2d}{d-2}\frac{q_{k-1}-\sigma}{2}=\theta^kq_0-\sigma[\theta^k+\theta^{k-1}+\cdots+\theta]=\theta^kq_0-\theta\sigma\frac{\theta^{k}-1}{\theta-1},
$$
with $\theta:=\frac{d}{d-2}>1$. Note that, since $q_0>\frac{\sigma}{\theta-1}$ (because $q_0>\frac{\sigma d}{2}$), $q_k$ grows exponentially fast to infinity as $k\to \infty$. By equation \eqref{eq:stepk}, for all $k\in\mathbb{N}$ we have that
\begin{align*}
	\|u(t_k)\|_{L^{q_k}(\Omega)}&\le \(\frac{2^{(k-1)/2}\tilde C_1(q_{k-1})}{\sqrt{T}}\)^{\frac{2}{\bar q_{k-1}}}\| u(t_{k-1})\|^{\frac{q_{k-1}}{\bar q_{k-1}}}_{L^{q_{k-1}}(\Omega)}\\
	&\le \prod_{i=0}^{k-1}\(\frac{2^{i/2}\tilde C_1(q_{i})}{\sqrt{T}}\)^{\frac{2}{\bar q_{i}}\prod_{j=i+1}^{k-1}\frac{q_j}{\bar q_j}}\| u(t_0)\|^{\prod_{i=0}^{k-1}\frac{q_{i}}{\bar q_{i}}}_{L^{q_{0}}(\Omega)},
\end{align*}
where $\bar q_{k}:=q_{k}-\sigma$ and where, by convention, $\prod_{j=i+1}^{k-1} \frac{q_j}{\bar q_j} =1$ if $k \leq i+1$. Letting $k\to \infty,$ by the exponential growth of $(q_k)_k$ we find a time $t_\infty \in (t_0,2T)$ such that
\begin{align*}
	\|u(t_\infty)\|_{L^{\infty}(\Omega)}&\leq \lim_{k\to \infty}\biggl( \prod_{i=0}^{k-1}\(\frac{2^{i/2} \tilde C_1(q_i)}{\sqrt{T}}\)^{\frac{2}{\bar q_{i}}\prod_{j=i+1}^{k-1}\frac{q_j}{\bar q_j}}
\biggr)\| u(t_0)\|^{\prod_{i=0}^{\infty}\frac{q_{i}}{\bar q_{i}}}_{L^{q_{0}}(\Omega)}\\
	&\leq \lim_{k\to \infty} \prod_{i=0}^{k-1}\(\frac{2^{i/2}\tilde C_1(q_i)}{\sqrt{T}}\)^{\frac{2}{\bar q_{i}}\prod_{j=0}^{\infty}\frac{q_j}{\bar q_j}}\| u(t_0)\|^{\prod_{i=0}^{\infty}\frac{q_{i}}{\bar q_{i}}}_{L^{q_{0}}(\Omega)}\\
	&= \(\frac{\prod_{i=0}^{\infty}[2^{i/2} \tilde C_1(q_i)]^{\frac{2}{\bar q_i}}}{\prod_{i=0}^{\infty}T^{1/\bar q_i}}\)^{\prod_{j=0}^{\infty}\frac{q_j}{\bar q_j}}
\| u(t_0)\|^{\prod_{i=0}^{\infty}\frac{q_{i}}{\bar q_{i}}}_{L^{q_{0}}(\Omega)}\\
	&=\(\frac{C_{1,\infty}}{T^{A_\infty}}\)^{B_\infty}
\| u(t_0)\|^{B_\infty}_{L^{q_{0}}(\Omega)},
\end{align*}
where
$$
	A_\infty:=\sum_{i=1}^\infty \frac{1}{\bar q_i},\qquad B_\infty:=\prod_{i=0}^{\infty}\frac{q_{i}}{\bar q_{i}},
\qquad C_{1,\infty}:=\prod_{i=0}^{\infty} \left( 2^{i/2} \tilde C_1(q_i)\right)^{\frac{2}{\bar q_i}} >0,
$$
are finite constants since there exist constants $c,C>0$ such that, for all $i \in\mathbb{N}$, $\bar q_i\geq c\, \theta^i$ and $\frac{q_{i}}{\bar q_{i}}\leq 1+C\theta^{-i}$.
Recalling the computation in \eqref{eq:der Lq}, we know that the $L^{q_0}$ and $L^\infty$ norms of $u$ are nonincreasing in time, and we thus conclude that
$$
\|u(3T)\|_{L^{\infty}(\Omega)}
\leq \|u(t_\infty)\|_{L^{\infty}(\Omega)}\leq \frac{C_{1,\infty}^{B_\infty}}{T^{A_\infty B_\infty}}
\| u(t_0)\|^{B_\infty}_{L^{q_{0}}(\Omega)} \leq \frac{C_{1,\infty}^{B_\infty}}{T^{A_\infty B_\infty}}
\| u_0\|^{B_\infty}_{L^{q_{0}}(\Omega)},
$$
that is, with the notation given in the statement of the theorem,
\be\label{eq:result-u}
	\|u(3T)\|_{L^{\infty}(\Omega)}\leq \frac{C_{1,\infty}^\beta}{T^\alpha} \| u_0\|^\beta_{L^{q}(\Omega)},
\ee
Since $f(3T)\leq u(3T)/\lambda$ and $u_0\leq f_0/\lambda$, we recover \eqref{eq:estimate-nu} and $C_1 = C_{1,\infty}^\beta/\lambda^{1+\beta}$.
\medskip

{\it Step 3: proof of \eqref{eq:estimate-nu2} ($\nu =-1$).} We proceed very similarly as in the case $\nu =1$ (Step 2), although in fact we use the result for $\nu=1$ as follows. By \eqref{eq:result-u} and the arbitrariness of $T$, we know that for every $t \in [\bar t,3T]$,
\bes
	u(t) \leq \frac{C_{1,\infty}^\beta}{(\bar t/3)^\alpha} \|u_0\|_{L^q(\Omega)}^\beta \leq \frac{C_{1,\infty}^\beta}{(T/3)^\alpha} \|u_0\|_{L^q(\Omega)}^\beta,
\ees
since $\bar t \geq T$.
Hence, with $\nu=-1$,
\bes
	u(\bar t)^{(\nu-1)\sigma} = u(\bar t)^{-2\sigma} \geq \left( \frac{C_{1,\infty}^\beta}{(T/3)^\alpha} \|u_0\|_{L^q(\Omega)}^\beta \right)^{-2\sigma} := \hat C^2 \|u_0\|_{L^q(\Omega)}^{-2\beta\sigma}T^{2\alpha\sigma},
\ees
where $\hat C = 9^{\alpha\sigma}C_{1,\infty}^{-\beta\sigma}$.
Therefore, coming back to \eqref{eq:H1} with $\nu=-1$ yields
\be\label{eq:H1-2}
	\hat C^2\|u_0\|_{L^q(\Omega)}^{-2\sigma\beta}T^{2\alpha\sigma} \int_{\Omega}\left|\nabla_x\bigl(u(\bar t,x)^{-\eta}\bigr) \right|^2 \di x \le\frac{1}{\lambda^2\,C_{-1}(q)^2\,T}\int_{\Omega}u_0(x)^{- q}\di x ,
\ee
Because $u_0^{-1} \in L^q(\Omega)$, \eqref{eq:H1-2} and \eqref{eq:L2} imply that $u(\bar t)^{-\eta} \in H^1(\Omega)$. Then, proceeding analogously as in Step 2, we get to
\be\label{eq:tildeC}
	\(\int_{\Omega}u(\bar t,x)^{2^*\eta}\di x\)^{1/2^*} \leq \frac{\tilde C_{-1}(q)}{\|u_0\|_{L^q(\Omega)}^{-\sigma\beta}T^{\alpha\sigma} \sqrt T}\(\int_{\Omega}\,u_0(x)^{q}\di x\)^{1/2},
\ee
where
\bes
	\tilde C_{-1}(q) = \frac{C_\mathrm{S}}{\lambda} \( \frac{|\Omega|^{-\frac\sigma2}}{(\lambda/M)^{\frac\sigma2}} + \frac{1}{\hat C C_{-1}(q)} \) >0.
\ees
In fact, in order to get to \eqref{eq:tildeC}, the only main difference with respect to Step 1 is the treatment of \eqref{eq:M}: we use Jensen's inequality to yield
\begin{align*}
	\|u(t_0)^{-1}\|_{L^q(\Omega)}^q&\geq |\Omega|^{1-q} \|u(t_0)^{-1}\|_{L^1(\Omega)}^q=|\Omega|^{1-q}\(\int_{\Omega} \left(\frac{f(t_0,x)}{m(x)}\right)^{-1} \di x\)^q\\
	&\geq |\Omega|^{1+q}\(\int_{\Omega} \frac{f(t_0,x)}{m(x)} \di x\)^{-q}\geq |\Omega|^{1+q}(\lambda/M)^q.
\end{align*}
Note that $\tilde C_{-1}(p) \to \frac{C_\mathrm{S}}{\lambda} \left( \frac{|\Omega|^{-\frac\sigma2}}{(\lambda/M)^{\frac\sigma2}} + \frac{1}{2\hat C\sqrt{r \sigma}} \right)$ as $p \to \infty$.

To initialise the iterative scheme we use the same notation as in the case $\nu=1$; we get
$$
	\|u(t_1)^{-1}\|^{\eta_0}_{L^{q_1}(\Omega)}\le \frac{\tilde C_{-1}(q_0)}{\|u(t_0)\|_{L^q(\Omega)}^{-\sigma\beta}T^{\alpha\sigma} \sqrt T}\| u(t_{0})^{-1}\|^{\frac{q_0}{2}}_{L^{q_0}(\Omega)}.
$$
We then follow the same strategy as in the case $\nu=1$. After $k\in\bb N$ iterations we obtain
\bes
	\|u(t_k)^{-1}\|_{L^{q_k}(\Omega)} \le \prod_{i=0}^{k-1}\(\frac{2^{i/2}\tilde C_{-1}(q_i)}{\|u_0\|_{L^q(\Omega)}^{-\sigma\beta}T^{\alpha\sigma} \sqrt T}\)^{\frac{2}{\bar q_{i}}\prod_{j=i+1}^{k-1}\frac{q_j}{\bar q_j}}\| u(t_0)^{-1}\|^{\prod_{i=0}^{k-1}\frac{q_{i}}{\bar q_{i}}}_{L^{q_{0}}(\Omega)}.
\ees
Letting $k\to \infty,$ because of the exponential growth of $(q_k)_k$ we find $t_\infty \in (0,3T)$ so that
\bes
	\|u(t_\infty)^{-1}\|_{L^{\infty}(\Omega)}\leq \(\frac{C_{-1,\infty}}{\(\|u_0\|_{L^q(\Omega)}^{-2\sigma\beta}T^{1+2\alpha\sigma}\)^{A_\infty}}\)^{B_\infty}
\| u(t_0)^{-1}\|^{B_\infty}_{L^{q_{0}}(\Omega)},
\ees
where $A_\infty$ and $B_\infty$ are as previously and
$$
	C_{-1,\infty}:=\prod_{i=0}^{\infty} \left( 2^{i/2} \tilde C_{-1}(q_i)\right)^{\frac{2}{\bar q_i}} >0,
$$
is a finite constant. By \eqref{eq:der Lq} we deduce that the $L^{q_0}$ and $L^\infty$ norms of $u^{-1}$ is nonincreasing, and we thus conclude that
\begin{align*}
	\|u(3T)^{-1}\|_{L^{\infty}(\Omega)} &\leq \|u(t_\infty)^{-1}\|_{L^{\infty}(\Omega)}\leq \frac{C_{-1,\infty}^{B_\infty}}{\(\|u_0\|_{L^q(\Omega)}^{-2\sigma\beta}T^{1+2\alpha\sigma}\)^{A_\infty B_\infty}}
\| u(t_0)^{-1}\|^{B_\infty}_{L^{q_{0}}(\Omega)}\\
	&\leq \frac{C_{-1,\infty}^{B_\infty}}{\(\|u_0\|_{L^q(\Omega)}^{-2\sigma\beta}T^{1+2\alpha\sigma}\)^{A_\infty B_\infty}}
\| u_0^{-1}\|^{B_\infty}_{L^{q_{0}}(\Omega)},
\end{align*}
that is, with the notation given in the statement of the theorem,
\bes
	\|u(3T)^{-1}\|_{L^{\infty}(\Omega)}\leq \frac{C_{-1,\infty}^\beta}{T^{\alpha(1+2\alpha\sigma)}}  \| u_0\|^{2\alpha\sigma\beta}_{L^{q}(\Omega)}\| u_0^{-1}\|^\beta_{L^{q}(\Omega)}.
\ees
Since it holds that $u(3T) \leq f(3T)/\lambda$, $f_0 \leq u_0/\lambda$ and $u_0\leq f_0/\lambda$, we finally recover \eqref{eq:estimate-nu2} and $C_{-1} = C_{-1,\infty}^\beta/ \lambda^{1+(1 + 2\alpha\sigma)\beta}$, which ends the proof.
 \end{proof}
  
  \begin{remark}
	The proof of Proposition \ref{prop:reg1} has been completely approached via the continuous-time study of the equation. In fact, it is also possible to obtain similar estimates also via the JKO iterations, using the flow-interchange technique (as for the $H^1$ estimates already presented in Section \ref{subsec:H1}). Yet, there are some drawbacks to the flow-interchange approach: it requires geodesic convexity of the functional, which means that it can only be used for positive powers (negative powers are only geodesically convex in dimension 1) and that it would be suitable to suppose that $\log m$ be concave; also, it does not allow to iterate infinitely many times, which finally provides estimates on the norms $\|u\|_{L^{p(\tau)}}$ for an expression $p(\tau)$ with $\lim_{\tau\to 0}p(\tau)=+\infty$. We decided to avoid this computation, because of its limited interest. \end{remark}

 \section{Long-time behaviour: proof of Theorem \ref{thm:bddness-longtime2}} \label{sec:long-time-behaviour}
 
 Thanks to the regularisation result of Section \ref{sec:reg}, we can now prove the first long-time convergence statement of Theorem \ref{thm:bddness-longtime2} (i.e., the $L^2$ convergence), which we restate below:
\begin{thm}[Exponential convergence to equilibrium]\label{L^2exp}
	Suppose that $f_0$ satisfies all the assumptions of Proposition \ref{prop:reg1} and that $f$ is a weak solution of \eqref{eq:f} starting from $f_0$. Then there exist constants $C,c>0$, independent of time, such that for all $t\geq0$ we have
$$
	\|f(t) - M\gamma m\|_{L^2(\Omega)} \leq Ce^{-ct}.
$$
\end{thm}
\begin{proof}
	By the computation in \eqref{eq:der Lq} we know that $f(t) \in L^2(\Omega)$ for all $t\geq0$, and, because $m$ is bounded on a bounded domain, we also have $m\in L^2(\Omega)$. This implies that for any $c>0$ and any $t_0>0$ there exists a constant $c_1>0$ such that
\bes
	\|f(t) - M\gamma m\|_{L^2(\Omega)} \leq c_1e^{-t} \quad \mbox{for all $t\in [0,t_0]$}.
\ees
By the instantaneous regularisation proved in Proposition \ref{prop:reg1}, together with the maximum principle given in Corollary \ref{cor:max princ}, it follows that there exist constants $c_0, C_0>0$ such that, for all $t>t_0$,
$$
	c_0m \leq f(t)\leq C_0m.
$$
This allows us to apply the very same argument as in \cite[Section 3]{Iac2} to obtain that there exist constants $c_2,c_3>0$ so that
\bes
	\|f(t) - M\gamma m\|_{L^2(\Omega)} \leq c_2e^{-c_3t} \quad \mbox{for all $t > t_0$}.
\ees
By choosing $C=\max(c_1,c_2)$ and $c=\min(1,c_3)$ we get the desired result.
\end{proof} 
 
\begin{remark}
	If $\rho\equiv1$, then the convergence estimate in the above theorem reads
\begin{equation*}
	\left\|f(t) - \frac{M}{|\Omega|}\right\|_{L^2(\Omega)} \leq Ce^{-ct}\quad \mbox{for all $t\geq0$}.
\end{equation*}
\end{remark} 

We now go on studying long-time $BV$ estimates thanks to Lemma \ref{lem:BV}. For convenience, we recall the definition of the $BV$ norm weighted by $m$: if $u\in BV(\Omega)$, we set $\|u\|_{BV(\Omega;m)}:=\int_\Omega m \di|\nabla u|.$ First, we establish the following estimate.

\begin{lem}\label{firstBVexp}
Suppose $D^2(\log m)\leq \Lambda\, {\rm Id}$. Suppose that $f_0$ satisfies all the assumptions of Proposition \ref{prop:reg1} and that $f$ is a weak solution of \eqref{eq:f} starting from $f_0$. Then, for all $t>0$ we have $f(t)/m\in BV(\Omega)$ and for every $t_0>0$ there exists a constant $C_2>0$ such that for $t_1>t_0$ we have 
$$ \left\| f(t_1)\right\|_{BV(\Omega;m)}\leq e^{C_2\Lambda(t_1-t_0)} \left\| f(t_0)\right\|_{BV(\Omega;m)}.$$
	\end{lem}
 \begin{proof}
Using $k$ iterated times the bound in Lemma \ref{lem:BV} (as it can be understood from the beginning of Section \ref{subsec:existence}), we get, for $\tau>0$,
  $$
  \left\| \frac{f_k^{(\tau)}}{m}\right\|_{BV(\Omega;m)}\leq \(\frac{1}{1-C_1\Lambda \tau}\)^k \left\| \frac{f_0}{m}\right\|_{BV(\Omega;m)},
 $$
 where $C_1$ is as in Lemma \ref{lem:BV}. For the limit of the JKO scheme, this implies that for some $C_2>0$ we get, for all $t\geq0$,
 $$ 
 	\|f(t)\|_{BV(\Omega;m)}\leq e^{C_2\Lambda t} \|f_0\|_{BV(\Omega;m)},
$$
as soon as $f_0\in BV(\Omega)$ and $c_0m\leq f_0\leq C_0m$ (use Lemma \ref{thm:existence-weak-1}). This can be translated into the desired bound  $\|f(t_1)\|_{BV(\Omega;m)}\leq e^{C_2\Lambda(t_1-t_0)} \|f(t_0)\|_{BV(\Omega;m)}$ for any $t_0>0$ and $t_1>t_0$, as soon as $f(t_0)$ is in $BV(\Omega)$ and is bounded from below and above. (We need to restart a JKO scheme from $f(t_0)$, which the uniqueness allows us to do.) Yet, the $L^2$ integrability of the $H^1$ norm of $u=f/m$ implies that $u(t)$ is in $H^1(\Omega)$, and hence in $BV(\Omega)$, for almost every positive time $t$, and the instantaneous regularisation given by Theorem \ref{thm:bddness-longtime} provides the lower and upper bounds, which finally gives the desired result.
\end{proof}

We finally show the second long-time convergence statement of Theorem \ref{thm:bddness-longtime2} (i.e., the $BV$ convergence), which we restate below:

\begin{thm}
	Suppose $D^2(\log m)\leq \Lambda\, {\rm Id}$. Suppose that $f_0$ satisfies all the assumptions of Proposition \ref{prop:reg1} and that $f$ is a weak solution of \eqref{eq:f} starting from $f_0$. Then, there are constants $C,c>0$ such that, for $t$ large enough, setting $u=f/m$, we have
$$ 
	\|u(t)\|_{BV(\Omega)}\leq Ce^{-ct},
$$
which implies $\|f-M\gamma m\|_{BV(\Omega)}\leq Ce^{-ct},$
possibly for a different constant $C$.
	\end{thm}
 \begin{proof}
 In case $\Lambda<0$, the result is just a simple consequence of Lemma \ref{firstBVexp}.
 For $\Lambda\geq 0$, consider the function $h(t)=\|f(t)\|_{BV(\Omega;m)}^2$ for all $t\geq0$. From Lemma \ref{firstBVexp}, for $t\geq 1$ we infer that for some $C>0$ it holds that
 $$
	h(t)\leq C\int_{t-1}^t h(s) \di s.
$$
 Moreover, $\int_1^\infty h(s)\di s<\infty$ since the $BV$ norm can be bounded with the $H^1$ norm. From 
 $$
	h(t)\leq C\int_{t-1}^{\infty} h(s) \di s, \qquad t\geq 1,
$$
we obtain $\lim_{t\to\infty} h(t)=0.$ Improving this from \eqref{H1estimate-semicont2} (or \eqref{H1estimate-semicont}) and comparing the $BV$ norm to the $H^1$ norm, we have 
$$
	h(t)\leq C \left(\G_{(r+3)}[f(t-1)]-\inf\left\{ \G_{(r+3)}[\mu]\;:\; \mu\in\mathcal M_M\right\}\right).
$$
Jensen's inequality together with the convexity of the power $r+3$ provides 
$$
	\inf\left\{ \G_{(r+3)}[\mu]\;:\; \mu\in\mathcal M_M\right\}=\G_{(r+3)}[M\gamma m]=M^{r+3}\gamma^{r+2}.
$$
Since for large $t$ the function $f(t)$ is bounded, we can use the Taylor expansion
 $$
	u^{r+3}\leq (M\gamma)^{r+3}+(r+3)(M\gamma)^{r+2}(u-M\gamma ) +C|u-M\gamma|^2
$$
and apply it to $u=f/m$. Since $M\gamma\int_\Omega m=M= \int_\Omega f=\int_\Omega um$ we have $\int_\Omega (u-M\gamma)m=0$, hence
 $$\G_{(r+3)}[f]- \G_{(r+3)}[M\gamma m]\leq C\int_\Omega |f-M\gamma m|^2.$$
 The $L^2$ exponential convergence result presented in Theorem \ref{L^2exp} allows us to conclude.
  \end{proof}

\begin{remark}
Because of the $L^\infty$ bounds from above and from below, the long-time $L^2$ convergence easily implies $L^p$ convergence for every $p\geq1$, and this convergence is still exponential. On the other hand, getting uniform convergence is a delicate matter: in dimension one it is trivial when $BV$ convergence is guaranteed, in higher dimension it is not.
\end{remark}

 {\bf Conflict of interest:} the Authors declare that they have no conflict of interest.
 \bigskip
 
 {\bf Acknowledgements:} MI thanks Matteo Bonforte for his comments to a preliminary version of this manuscript and Jos\'e A. Carrillo for interesting discussions on this topic during a visit at Imperial College in 2015. FSP thanks the CNA at Carnegie Mellon University for their kind support. FS thanks Imperial College for their warm hospitality in 2017 when this work started, via a CNRS-Imperial fellowship.

\end{document}